\documentclass[12pt, reqno]{amsart}
\usepackage{amsmath, amsthm, amscd, amsfonts, amssymb, graphicx, color}
\usepackage[bookmarksnumbered, colorlinks, plainpages]{hyperref}
\hypersetup{colorlinks=true,linkcolor=red, anchorcolor=green, citecolor=cyan, urlcolor=red, filecolor=magenta, pdftoolbar=true}

\newtheorem{theorem}{Theorem}[section]

\newtheorem{corollary}[theorem]{Corollary}
\theoremstyle{definition}
\newtheorem{definition}[theorem]{Definition}
\newtheorem{example}[theorem]{Example}

\theoremstyle{remark}
\newtheorem{remark}[theorem]{Remark}
\numberwithin{equation}{section}

\begin{document}

\setcounter{page}{1}

\title[Disjoint hypercyclic]{Disjoint hypercyclic weighted pseudo-shift operators generated by different shifts}

\author[Y. Wang, \MakeLowercase{and} Z.H. Zhou]{Ya Wang,$^1$ \MakeLowercase{and} Ze-Hua Zhou$^2$$^{*}$}

\address{$^{1}$School of Mathematics, Tianjin University, Tianjin 300354, P.R. China.}
\email{\textcolor[rgb]{0.00,0.00,0.84}{wangyasjxsy0802@163.com}}

\address{$^{2}$School of Mathematics, Tianjin University, Tianjin 300354, P.R. China.}
\email{\textcolor[rgb]{0.00,0.00,0.84}{zehuazhoumath@aliyun.com;zhzhou@tju.edu.cn}}


\let\thefootnote\relax\footnote{Copyright 2016 by the Tusi Mathematical Research Group.}



\date{Received: xxxxxx; Revised: yyyyyy; Accepted: zzzzzz.
\newline \indent $^{*}$Corresponding author.\\
The work was supported in part by the National Natural Science Foundation of
China (Grant Nos. 11771323; 11371276).}

\begin{abstract}
Let $I$ be a countably infinite index set, and let $X$ be a Banach sequence space over $I.$ In this article, we characterize disjoint hypercyclic and supercyclic weighted pseudo-shift operators on $X$ in terms of the weights, the OP-basis, and the shift mappings on $I.$ Also, the shifts on weighted $L^p$ spaces of a directed tree and the operator weighted shifts on $\ell^2(\mathbb{Z,\mathcal{K}})$ are investigated as special cases.
\end{abstract} \maketitle

\section{Introduction and preliminaries}

The study of disjointness in hypercyclicity was initiated in 1970 by Bernal-Gonz\'{a}lez \cite{B1} and B\`{e}s and Peris \cite{BP}, respectively. Since then, disjoint hypercyclicity was investigated by many authors, we recommend \cite{BMP}, \cite{BMS}, \cite{BMPS}, \cite{Sh} and \cite{SA}  for recent works on this subject.

The new notions, disjoint hypercyclic operator and disjoint supercyclic operator are derived from the much older notions of hypercyclic operator and supercyclic operator in linear dynamics. Let $X$ be a separable infinite dimensional complex Banach space, we denote by $L(X)$ the set of all continuous and linear operators on $X.$ An operator $T\in L(X)$ is said to be \emph{hypercyclic} if there is some vector $x\in X$ such that the \emph{orbit} $\mathrm{Orb}(T,x)=\{T^nx : n\in \mathbb{N}\}$ (where $\mathbb{N} = \{0, 1, 2, 3, \ldots\}$) is dense in $X$. In such a case, $x$ is called a \emph{hypercyclic vector} for $T.$ Similarly, $T$ is said to be \emph{supercyclic} if there exists an $x\in X$ such that $\mathbb{C}\cdot\mathrm{Orb}(T,x)=\{\lambda T^nx : n\in \mathbb{N}, \lambda\in \mathbb{C}\}$ is dense in $X.$ For the background about hypercyclicity and supercyclicity we refer to the excellent monographs by Bayart and Matheron \cite{BM} and by Grosse-Erdmann and Peris Manguillot \cite{GM}.

 $N \geq 2,$ hypercyclic (respectively, supercyclic) operators $T_1,  T_2, \ldots, T_N $ acting on the same space $X$ are said to be \emph{disjoint} or \emph{d-hypercyclic}(respectively, \emph{d-supercyclic}) if their direct sum $\oplus_{m=1}^{N}T_m$ has a hypercyclic (respectively, supercyclic) vector of the form $(x, x, \cdots, x)$ in $X^N.$ $x$ is called a \emph{d-hypercyclic}(respectively, \emph{d-supercyclic}) vector for $T_1,  T_2, \ldots, T_N.$ If the set of d-hypercyclic (respectively, d-supercyclic) vectors is dense in $X,$ we say $T_1, T_2, \ldots, T_N$ are
 \emph{densely d-hypercyclic}(respectively, \emph{densely d-supercyclic}).

 In the study of linear dynamics, one large source of examples is the class of weighted shifts. In \cite{S1} and \cite{SH}, Salas characterized the hypercyclic and supercyclic weighted shifts on $\ell^p(\mathbb{Z}) \; (1\leq p < \infty)$ respectively.
 The characterizations for weighted shifts on $\ell^p(\mathbb{Z})\; (1\leq p < \infty)$ to be disjoint hypercyclic and disjoint supercyclic were provided in \cite{BP}, \cite{MO} and \cite{LZ}. As generalizations of weighted shifts on $\ell^p(\mathbb{Z})\; (1\leq p < \infty)$, Grosse-Erdmann \cite{GE} studied the hypercyclicity of weighted pseudo-shifts on F-sequence spaces.  Hazarika and Arora considered the hypercyclic operator weighted shifts on $\ell^2(\mathbb{Z},\mathcal{K})$ in \cite{HA}. And the equivalent conditions for the weighted pseudo-shifts and operator weighted shifts to be supercyclic were obtained in \cite{LZ2}. Inspired by these results, in \cite{WZ} we characterized the disjoint hypercyclic powers of weighted pseudo-shifts as below.

 \begin{theorem}\label{d-hyper1}
 $[20, \mbox{ Theorem } 2.1]$
Let $I$ be a countably infinite index set and $X$ be a Banach sequence space over $I$, in which $(e_i)_{i\in I}$ is an OP-basis. Let $\varphi : I \rightarrow I$ be an injective map. $N\geq 2,$ for each $1 \leq l \leq N$, $T_{l} = T_{b^{(l)}, \varphi} : X \rightarrow X$ be a weighted pseudo-shift generated by $\varphi$ with weight sequence $b^{(l)} = (b_i^{(l)})_{i\in I}.$ Then for any integers $1\leq r_1 < r_2 < \cdots\cdots < r_{N},$ the following are equivalent:

\begin{enumerate}
\item $T_{1} ^{r_1}, T_{2} ^{r_2}, \ldots, T_{N} ^{r_N}$ are densely d-hypercyclic.

\item $(\alpha)$ The mapping $\varphi : I\rightarrow I$ has no periodic points.

     \quad  \; $(\beta)$ There exists an increasing sequence $(n_k)_{k\geq 1}$ of positive integers such that for every $i\in I$ we have:

    (H1) $\mbox{ If } 1 \leq l \leq N,$

   \begin{eqnarray*}
   \left\{\begin{array}{ll}
   \left( \prod\limits_{v = 0}^{r_l n_k -1} b_{\varphi^v(i)} ^{(l)}\right)^{-1} e_{\varphi^{r_l n_k}(i)} \rightarrow 0,\\
   \\
   \left( \prod\limits_{v = 1}^{r_l n_k} b_{\psi^v(i)} ^{(l)}\right) e_{\psi^{r_l n_k}(i)} \rightarrow 0
   \end{array}\right.
   \mbox{ in } X, \mbox{ as } k\rightarrow \infty.
   \end{eqnarray*}

(H2) If $1 \leq s < l \leq N,$
\begin{eqnarray*}
   \left\{\begin{array}{ll}
   \left( \prod\limits_{v = 0}^{r_s n_k -1} b_{\varphi^v(i)} ^{(s)}\right)^{-1}
   \left( \prod\limits_{v = 1}^{r_l n_k} b_{\psi^v (\varphi^{r_s n_k}(i))}^{(l)} \right) e_{\psi^{(r_l-r_s) n_k}(i)} \rightarrow 0,\\
   \\
    \left( \prod\limits_{v = 0}^{r_l n_k -1} b_{\varphi^v(i)} ^{(l)}\right)^{-1}
   \left( \prod\limits_{v = 1}^{r_s n_k} b_{\psi^v (\varphi^{r_l n_k}(i))} ^{(s)}\right) e_{\varphi^{(r_l-r_s) n_k}(i)} \rightarrow 0\\
   \end{array}\right.
   \mbox{ in } X, \mbox{ as } k\rightarrow \infty.
   \end{eqnarray*}

   \item  $T_1^{r_1 },  T_{2} ^{r_2}, \ldots,T_{N }^{r_N }$ satisfy the d-Hypercyclicity Criterion.
\end{enumerate}
\end{theorem}
\begin{remark}
In paper \cite{WZ}, Theorem \ref{d-hyper1} holds for weighted pseudo-shifts on an F-sequence space.
\end{remark}
Note that, the weighted pseudo-shifts $T_{b^{(1)}, \varphi}, T_{b^{(1)}, \varphi},\ldots, T_{b^{(N)}, \varphi}$ in the above theorem are determined by different weights $b^{(l)}\;(1 \leq l \leq N)$ and the same injective mapping $\varphi$. In this article, we continue our research by considering the disjoint hypercyclicity and disjoint supercyclicity of finite weighted pseudo-shifts that are generated by different weights and different injective mappings.

   The following criteria play an important role in our main results. The first criterion is due to B\`{e}s and Peris \cite{BP} and the second one is due to Martin \cite{MO}.
\begin{definition} \label{d-hypercriterion}
Let $(n_k)_k$ be a strictly increasing sequence of positive integers. We say that $T_{1}, T_{2}, \ldots, T_{N} \in L(X)$ satisfy the d-Hypercyclicity Criterion with respect to $(n_k)_k$ provided there exist dense subsets $X_0, X_1, \ldots, X_N$ of $X$ and mappings $S_{l,k} : X_l \rightarrow X (1 \leq l \leq N, k\in \mathbb{N})$ satisfying
\begin{eqnarray*}
T_l^{n_k} \;\xrightarrow[k\rightarrow \infty]{}\;\; 0 \;\;\; \mbox{pointwise on } X_0,
\end{eqnarray*}
\begin{eqnarray*}
S_{l,k}\;\;\;\xrightarrow[k\rightarrow \infty]{} 0 \;\;\; \mbox{pointwise on } X_l, \mbox{ and }
\end{eqnarray*}
\begin{eqnarray}\label{1.1}
(T_l^{n_k}S_{i,k} - \delta _{i,l} Id_{X_i})\;\;\;\xrightarrow[k\rightarrow \infty]{} 0 \;\;\; \mbox{pointwise on } X_i\; (1\leq i\leq N).
\end{eqnarray}
In general,we say that $T_{1}, T_{2}, \ldots, T_{N} $ satisfy the d-Hypercyclicity Criterion if there exists some sequence $(n_k)_k$ for which \eqref{1.1} is satisfied.
\end{definition}

If $T_{1}, T_{2}, \ldots, T_{N} $ satisfy the d-Hypercyclicity Criterion with respect to a sequence $(n_k)_k$. Then $T_{1}, T_{2}, \ldots, T_{N} $ are densely d-hypercyclic.

\begin{definition} \label{d-supercriterion}
Let $X$ be a Banach space, $(n_k)_k$ be a strictly increasing sequence of positive integers and $N \geq 2$. We say that $T_{1}, T_{2}, \ldots, T_{N} \in L(X)$ satisfy the d-Supercyclicity Criterion with respect to $(n_k)_k$ provided there exist dense subsets $X_0, X_1, \ldots, X_N$ of $X$ and mappings $S_{l,k} : X_l \rightarrow X (1 \leq l \leq N, k\in \mathbb{N})$ such that for $1\leq l\leq N,$
\begin{description}
\item[(i)]$(T_l^{n_k}S_{i,k} - \delta _{i,l} Id_{X_i})\;\;\;\xrightarrow[k\rightarrow \infty]{} 0 \;\;\; \mbox{pointwise on } X_i\; (1\leq i\leq N);$
\item[(ii)]$\lim\limits_{k\rightarrow \infty}\left\|T_l^{n_k}x\right\| \cdot \left\|\sum\limits_{j=1}^{N} S_{j,k} y_j \right\|=0 \;\;\; \mbox{for } x\in X_0, y_j\in  X_j.$
\end{description}
\end{definition}

Let $N \geq 2$ and $T_{1}, T_{2}, \ldots, T_{N}\in L(X)$ satisfy the d-Supercyclicity Criterion. Then $T_{1}, T_{2}, \ldots, T_{N} $ have a dense set of d-supercyclic vectors.

To proceed further we recall some terminology about the sequence spaces and the weighted pseudo-shifts. For a comprehensive survey we recommend Grosse-Erdmann's paper \cite{GE}.

\begin{definition}
$\textbf{(Sequence\;Space)}$\;\; If we allow an arbitrary countably infinite set $I$ as an index set, then a \emph{sequence space over} $I$ is a subspace of the space $\omega(I)=\mathbb{C}^{I}$ of all scalar families $(x_i)_{i\in I}.$ The space $\omega(I)$ is endowed with its natural product topology.

A \emph{topological sequence space $X$ over $I$} is a sequence space over $I$ that is endowed with a linear topology in such a way that the inclusion mapping $X\hookrightarrow \omega(I)$ is continuous or, equivalently, that every \emph{coordinate functional} $f_i: X\rightarrow \mathbb{C}, (x_k)_{k\in I}\mapsto x_i(i \in I)$ is continuous. A \emph{Banach $($Hilbert, F-$)$ sequence space over $I$} is a topological sequence space over $I$ that is a Banach $(\mbox{Hilbert, F-})$ space.
\end{definition}

\begin{definition}
$\textbf{(OP-basis)}$\;\;
 By $(e_i)_{i\in I}$ we denote the canonical unit vectors $e_i=(\delta_{ik})_{k\in I}$ in a topological sequence space $X$ over $I.$ We say $(e_i)_{i\in I}$ is an $OP-basis$ or ($Ovsepian\; Pe{\l}czy\acute{n}ski \; basis$) if $\mbox{span}\{e_i: i\in I\}$ is a dense subspace of $X$ and
  the family of  $coordinate\;  projections$ $x\mapsto x_i e_i(i \in I)$ on $X$ is equicontinuous.

  Note that in a Banach sequence space over $I,$ the family of coordinate projections is equicontinuous if and only if $\sup_{i\in I}||e_i|| ||f_i||< \infty.$
\end{definition}

\begin{definition}
$\textbf{(Pseudo-shift Operator)}$\;\;
 Let $X$ be a Banach sequence
  space over $I$. Then a continuous linear operator $T : X \rightarrow X$ is called a \textit{ weighted pseudo-shift} if there is a sequence $(b_i)_{i\in I}$ of non-zero scalars and an injective mapping $\varphi : I \rightarrow I$ such that
 $$T(x_i)_{i\in I}=(b_i x_{\varphi(i)})_{i\in I}$$
 for $(x_i) \in X.$ We then write $T = T_{b,\varphi},$ and $(b_i)_{i\in I}$ is called the \textit{ weight sequence}.
\end{definition}

\begin{remark}
\begin{enumerate}
\item If $T = T_{b,\varphi} : X \rightarrow X$ is a weighted pseudo-shift, then
each $T^n (n\geq 1)$ is also a weighted pseudo-shift as follows
\[
T^{n}(x_i)_{i\in I} = (b_{n,i}x_{\varphi^n(i)})_{i\in I}
\]
where
 $$\varphi^n(i) =(\varphi \circ \varphi \circ \cdots \circ \varphi)(i)\;\;\; (\textit{n}- \mbox{fold})$$
\begin{eqnarray*}
b_{n,i}&=&b_ib_{\varphi(i)} \cdots b_{\varphi^{n-1}(i)} = \prod\limits_{v=0}^{n-1}b_{\varphi^{v}(i)}.
\end{eqnarray*}

\item We consider the inverse $\psi = \varphi^{-1} :\varphi(I) \rightarrow I$ of the mapping $\varphi.$ We also set
\[
b_{\psi(i)} = 0 \;\;\; \mbox{ and }\;\;\; e_{\psi(i)} = 0 \;\;\;  \mbox{if } i\in I\setminus \varphi(I),
\]
i.e. when $\psi(i)$ is `` undefined ". Then
for all $i \in I,$
$$T_{b,\varphi} e_i = b_{\psi(i)} e_{\psi(i)}.$$

\item We denote $\psi^n = \psi \circ \psi \circ \cdots \circ \psi$ (\textit{n}-fold), and we set $b_{\psi^{n}(i)} = 0 \mbox{ and } e_{\psi^{n}(i)} = 0$ when $\psi^{n}(i)$ is `` undefined ".
\end{enumerate}
\end{remark}

\begin{definition}
Let $\varphi: I \rightarrow I$ be a map on $I$ and let $(\varphi^n)_n$ be the sequence of iterates of the mapping $\varphi$ (that is, $\varphi^n : = \varphi \circ \varphi \circ \cdots \circ \varphi$ ($n$-fold)). We call
$(\varphi^n)_n$ is a \textit{run-away sequence} if for each pair of finite subsets $I_0 \subset I,$ there exists an $n_0 \in \mathbb{N}$ such that $\varphi^n(I_0)\cap I_0 = \emptyset$ for every $n \geq n_0.$

 Let $\varphi_1: I \rightarrow I$ and $\varphi_2 :I \rightarrow I$ be two maps on $I.$ We call the sequence $(\varphi_1^n)_n$ is run-away with $(\varphi_2^n)_n$ if for any finite subset $I_0 \subset I,$ there exists an $n_0 \in \mathbb{N}$ such that $\varphi_1^n(I_0)\cap \varphi_2^n(I_0) = \emptyset$ for every $n \geq n_0.$
\end{definition}

\section{Disjoint hypercyclic powers of weighted pseudo-shifts}

In this section, let $X$ be a Banach sequence space over $I.$ We will characterize disjoint hypercyclic weighted pseudo-shifts on $X$ generated by different weights and different injective maps. Which is a generalization of Theorem \ref{d-hyper1}.

The following is the main result in this section.
\begin{theorem}\label{main 1}
Let $X$ be a Banach sequence space over $I$, in which $(e_i)_{i\in I}$ is an OP-basis. Let integers $1\leq r_1 < r_2 < \cdots < r_{N}$ be given. $N\geq 2,$ for each $1 \leq l \leq N$, let $T_{l} = T_{b^{(l)}, \varphi_{l}} : X \rightarrow X$ be a weighted pseudo-shift with the weight sequence $b^{(l)} = (b_i^{(l)})_{i\in I}$ and the injective mapping $\varphi_{l}.$ If for any $1 \leq s < l \leq N,$ the sequence $((\varphi_{s}^{r_s})^n)_n$ is run-away with $((\varphi_{l}^{r_l})^n)_n$. Then the following are equivalent:

\begin{enumerate}
\item $T_{1} ^{r_1}, T_{2} ^{r_2}, \ldots, T_{N} ^{r_N}$ are densely d-hypercyclic.

\item\quad $(\alpha)$ For each $1 \leq l \leq N$, the mapping $\varphi_{l} : I\rightarrow I$ has no periodic points.

  \quad\;\;\;\;\;$(\beta)$ There exists an increasing sequence $(n_k)_{k\geq 1}$ of positive integers such that for every $i\in I$ we have:

    (H1) $\mbox{ If } 1 \leq l \leq N,$

   \begin{eqnarray*}
   \left\{\begin{array}{ll}
   \lim\limits_{k\rightarrow \infty}\left\|\left( \prod\limits_{v = 0}^{r_l n_k -1} b_{\varphi^v_{l}(i)} ^{(l)}\right)^{-1} e_{\varphi^{r_l n_k}_{l}(i)}\right\| = 0,\\
   \\
   \lim\limits_{k\rightarrow \infty}\left\|\left( \prod\limits_{v = 1}^{r_l n_k} b_{\psi^v_{l}(i)} ^{(l)}\right) e_{\psi^{r_l n_k}_{l}(i)}\right\| = 0
   \end{array}\right.
   \end{eqnarray*}

(H2) If $1 \leq s < l \leq N,$
\begin{eqnarray*}
   \left\{\begin{array}{ll}
   \lim\limits_{k\rightarrow \infty}\left\|\left( \prod\limits_{v = 0}^{r_s n_k -1} b_{\varphi^v_{s}(i)} ^{(s)}\right)^{-1}
   \left( \prod\limits_{v = 1}^{r_l n_k} b_{\psi^v_{l} (\varphi^{r_s n_k}_{s}(i))}^{(l)} \right) e_{\psi_{l}^{r_ln_k}(\varphi_{s}^{r_s n_k}(i))} \right\| = 0,\\
   \\
   \lim\limits_{k\rightarrow \infty}\left\| \left( \prod\limits_{v = 0}^{r_l n_k -1} b_{\varphi^v_{l}(i)} ^{(l)}\right)^{-1}
   \left( \prod\limits_{v = 1}^{r_s n_k} b_{\psi^v_{s} (\varphi^{r_l n_k}_{l}(i))} ^{(s)}\right) e_{\psi_{s}^{r_sn_k}(\varphi_{l}^{r_l n_k}(i))} \right\| = 0
   \end{array}\right.
   \end{eqnarray*}

   \item  $T_1^{r_1 },  T_{2} ^{r_2}, \ldots,T_{N }^{r_N }$ satisfy the d-Hypercyclicity Criterion.
   \end{enumerate}
\end{theorem}

\begin{proof}

$(1)\Rightarrow (2).$
 Since $T_{1} ^{r_1}, T_{2} ^{r_2}, \ldots, T_{N} ^{r_N}$ are d-hypercyclic, $T_{l}$ is hypercyclic for each $1 \leq l \leq N.$ In \cite{GE}, Grosse-Erdmann proved that if the weighted pseudo-shift $T_{b^{(l)}, \varphi_{l}}$ is hypercyclic, then the mapping $\varphi_{l} : I\rightarrow I$ has no periodic points, which gives that $(\varphi_{l}^n)_n$ is a run-away sequence.

 By assumption $I$ is a countably infinite set, we fix $$I : =\{i_1, i_2, \cdots, i_n,\cdots \}$$ and set $I_k: = \{i_1, i_2, \cdots, i_k\}$ for each integer $k$ with $k\geq 1.$
To complete the proof of $(\beta)$, we first show that for any integers $k, N_0$ with $k\geq 1$ and $N_0 \in \mathbb{\mathbb{N}},$ there is an integer $n_k > N_0$ such that for every $i\in I_k,$

  $\mbox{ if } 1 \leq l \leq N,$

   \begin{eqnarray} \label{a1}
   \left\{\begin{array}{ll}
 (i)\; \left\| \left( \prod\limits_{v = 0}^{r_l n_k -1} b_{\varphi^v_l(i)} ^{(l)}\right)^{-1} e_{\varphi^{r_l n_k}_l(i)}\right\| < \frac{1}{k},\\
 \\
   (ii)\; \;\; \left\|\left( \prod\limits_{v = 1}^{r_l n_k} b_{\psi^v_l(i)} ^{(l)}\right) e_{\psi^{r_l n_k}_l(i)}\right\| < \frac{1}{k},
   \end{array}\right.
   \end{eqnarray}

 if $1 \leq s < l \leq N,$
\begin{eqnarray}\label{a2}
   \left\{\begin{array}{ll}
   (i)\;\left\|\left( \prod\limits_{v = 0}^{r_s n_k -1} b_{\varphi^v_s (i)} ^{(s)}\right)^{-1}
   \left( \prod\limits_{v = 1}^{r_l n_k} b_{\psi^v_l (\varphi^{r_s n_k}_s(i))} ^{(l)}\right) e_{\psi_{l}^{r_ln_k}(\varphi_{s}^{r_s n_k}(i))}\right\| < \frac{1}{k}, \\
   \\
  (ii)\; \; \left\| \left( \prod\limits_{v = 0}^{r_l n_k -1} b_{\varphi^v_{l}(i)} ^{(l)}\right)^{-1}
   \left( \prod\limits_{v = 1}^{r_s n_k} b_{\psi^v_{s} (\varphi^{r_l n_k}_{l}(i))} ^{(s)}\right) e_{\psi_{s}^{r_sn_k}(\varphi_{l}^{r_l n_k}(i))}\right\| < \frac{1}{k}. \\
   \end{array}\right.
   \end{eqnarray}
Let integers $k\geq 1$ and $N_0 \in \mathbb{N}$ be given. Notice that $(e_i)_{i\in I}$ is an OP-basis, by the equicontinuity of the coordinate projections in $X,$ there is some $\delta_k > 0$ so that for $x = (x_i)_{i\in I} \in X,$
   \begin{eqnarray}\label{a3}
   ||x_i e_i|| < \frac{1}{2k} \;\; \mbox{ for all } i\in I, \mbox{ if } ||x|| < \delta_k.
    \end{eqnarray}
Since for each $1\leq l \leq N$ the sequence $(\varphi_{l}^n)_n$ is run-away and for $1\leq s < l \leq N,$  $((\varphi_{s}^{r_s})^n)_n$ is run-away with $((\varphi_{l}^{r_l})^n)_n$, there exists an integer $\widetilde{N_0}\in \mathbb{N}$ such that
\begin{eqnarray}\label{a4}
 \left\{\begin{array}{ll}
(i)\;\varphi^{r_ln}_l(I_k)\cap I_k = \emptyset \; (1\leq l \leq N),\\
\\
(ii)\;\psi_{l}^{r_ln}(\varphi_{s}^{r_s n}(I_k)\cap  \varphi^{r_ln}_l(I))\cap I_k = \emptyset \; (1\leq s < l \leq N)
\end{array}\right.
\end{eqnarray}
for all $n\geq \widetilde{N_0}.$

   By the density of d-hypercyclic vectors in $X$ and the continuous inclusion of $X$ into $\mathbb{K}^I,$  we can find a d-hypercyclic vector $x = (x_i)_{i\in I}\in X$ and an integer $n_k > \max\{N_0, \widetilde{N_0}\}$ such that
     \begin{eqnarray}\label{a5}
      \left\{\begin{array}{ll}
     (i) \; \|x-\sum_{i\in I_k}e_i\| < \delta_k,\\
     \;\\
     (ii) \; \sup\limits_{i\in I_k}|x_i-1|\leq \frac{1}{2}\\
     \end{array}\right.
   \end{eqnarray}
      and for each $1\leq l \leq N,$
      \begin{eqnarray}\label{a6}
      \left\{\begin{array}{ll}
     (i) \; \left\|y^{(l)} -\sum_{i\in I_k}e_i\right\| < \delta_k,\\
     \;\\
     (ii) \; \sup\limits_{i\in I_k}| y_i^{(l)}-1|\leq \frac{1}{2}, \\
     \end{array}\right.
     \end{eqnarray}
  where $y^{(l)} := T_l^{r_ln_k}x = (\prod\limits_{v=0}^{r_ln_k-1}b_{\varphi^v_l(i)}^{(l)}x_{\varphi^{r_ln_k}_l(i)})_{i\in I} = (y^{(l)}_i)_{i\in I}.$ \\
 By \eqref{a3}, the first inequality in \eqref{a5} implies that
\begin{eqnarray}\label{a7}
\|x_i e_i\| < \frac{1}{2k} \;\;\;\mbox{ if } i\notin I_k,
\end{eqnarray}
thus by the first inequality in \eqref{a4}, for each $1\leq l \leq N$ we have that
\begin{eqnarray}\label{a8}
 \left\|x_{\varphi^{r_ln_k}_l(i)}e_{\varphi^{r_ln_k}_l(i)}\right\| < \frac{1}{2k}\;\; \mbox{ for } i\in I_k.
 \end{eqnarray}
By the second inequality in \eqref{a6}, for each $i\in I_k$ and $1\leq l \leq N,$
\begin{eqnarray*}
\left|\prod\limits_{v=0}^{r_ln_k-1}b_{\varphi^v_l(i)}^{(l)}x_{\varphi^{r_ln_k}_l(i)}-1\right|\leq \frac{1}{2},
\end{eqnarray*}
which means that
\begin{eqnarray}\label{a9}
\left\{\begin{array}{ll}
(i)\;x_{\varphi^{r_ln_k}_l(i)} \neq 0, \\
\\
(ii)\;\left|\left(\prod\limits_{v=0}^{r_ln_k-1}b_{\varphi^v_l(i)}^{(l)}x_{\varphi^{r_ln_k}_l(i)}\right)^{-1}\right| \leq 2.\\
\end{array}\right.
\end{eqnarray}
For each $i\in I_k$ and $1\leq l \leq N,$ by \eqref{a8} and \eqref{a9}, it follows that
\begin{eqnarray*}
\left\| \left(\prod \limits_{v=0}^{r_ln_k-1} b_{\varphi^{v}_l(i)}^{(l)}\right)^{-1}e_{\varphi^{r_ln_k}_l(i)}\right\|& = & \left\|\frac{1}{\left(\prod \limits_{v=0}^{r_ln_k-1} b_{\varphi^{v}_l(i)}^{(l)}\right)x_{\varphi^{r_ln_k}_l(i)}} x_{\varphi^{r_ln_k}(i)}e_{\varphi^{r_ln_k}_l(i)}\right\|\\
& \leq &  2\left\|x_{\varphi^{r_ln_k}_l(i)}e_{\varphi^{r_ln_k}_l(i)}\right\| \\
& < & \frac{1}{k}.
\end{eqnarray*}

 We deduce from \eqref{a4} and the definition of $\psi$ that
 \begin{eqnarray}\label{a10}
 \left\{\begin{array}{ll}
\psi^{r_ln_k}_l(I_k\cap \varphi^{r_ln_k}_l(I))\cap I_k = \emptyset \;\; \mbox{ for }1\leq l \leq N,\\
\\
\\ \psi_{s}^{r_sn_k}(\varphi_{l}^{r_l n_k}(I_k)\cap  \varphi^{r_sn_k}_s(I))\cap I_k = \emptyset\;\; \mbox{ for }1\leq s <  l \leq N.
\end{array}\right.
\end{eqnarray}
By \eqref{a3}, the first inequality in \eqref{a6} implies
\begin{eqnarray}\label{a11}
\left\|\prod\limits_{v=0}^{r_ln_k-1}b_{\varphi^v_l(i)}^{(l)}x_{\varphi^{r_ln_k}_l(i)}e_i\right\| < \frac{1}{2k}\;\;\mbox{ for } i\notin I_k \mbox{ and } 1\leq l \leq N.
\end{eqnarray}
Note that for each $1\leq l \leq N,$
$$e_{\psi^{r_ln_k}_l(i)} = 0 \;\;\;\mbox{    if }\; i\in I_k\backslash \varphi^{r_ln_k}_l(I)$$
and for $1\leq s <  l \leq N,$
$$e_{\psi_{l}^{r_ln_k}(\varphi_{s}^{r_s n_k}(i))} = 0 \;\;\;\mbox{    if }\; i\in \varphi^{r_sn_k}_s(I_k)\backslash \varphi^{r_ln_k}_l(I),$$
$$e_{\psi_{s}^{r_sn_k}(\varphi_{l}^{r_l n_k}(i))} = 0 \;\;\;\mbox{    if }\; i\in \varphi^{r_ln_k}_l(I_k)\backslash \varphi^{r_sn_k}_s(I).$$
So by \eqref{a11}, \eqref{a10} and $(ii)$ of  \eqref{a4} we obtain that:

For each $i\in I_k$ and $1\leq l \leq N,$
\begin{eqnarray}\label{a12}
&\;&\left\|\prod\limits_{v=0}^{r_ln_k-1}b_{\varphi^v_l(\psi^{r_ln_k}_l(i))}^{(l)}x_{\varphi^{r_ln_k}_l(\psi^{r_ln_k}_l(i))}e_{\psi^{r_ln_k}_l(i)}\right\|\nonumber\\
&\;&=\left\|\prod\limits_{v=1}^{r_ln_k} b_{\psi^v_l(i)}^{(l)}x_ie_{\psi^{r_ln_k}_l(i)}\right\|< \frac{1}{2k}.
\end{eqnarray}

 For each $i\in I_k$ and $1\leq s < l \leq N,$
\begin{eqnarray}\label{a13}
&\;&\left\|\prod\limits_{v=0}^{r_ln_k-1}b_{\varphi^v_l(\psi_{l}^{r_ln_k}(\varphi_{s}^{r_s n_k}(i)))}^{(l)}x_{\varphi^{r_ln_k}_l(\psi_{l}^{r_ln_k}(\varphi_{s}^{r_s n_k}(i)))}e_{\psi_{l}^{r_ln_k}(\varphi_{s}^{r_s n_k}(i))}\right\|\nonumber\\
&\;&=\left\|\prod\limits_{v=1}^{r_ln_k}b_{\psi^v_l(\varphi_s^{r_sn}(i))}^{(l)}x_{\varphi^{r_sn_k}_s(i)}e_{\psi_{l}^{r_ln_k}(\varphi_{s}^{r_s n_k}(i))}\right\|
< \frac{1}{2k}
\end{eqnarray}
and
\begin{eqnarray}\label{a14}
&\;&\left\|\prod\limits_{v=0}^{r_sn_k-1}b_{\varphi^v_s(\psi_{s}^{r_sn_k}(\varphi_{l}^{r_l n_k}(i)))}^{(s)}x_{\varphi_s^{r_sn_k}(\psi_{s}^{r_sn_k}(\varphi_{l}^{r_l n_k}(i)))}e_{\psi_{s}^{r_sn_k}(\varphi_{l}^{r_l n_k}(i))}\right\|\nonumber\\
&\;&=\left\|\prod\limits_{v=1}^{r_sn_k}b_{\psi^v_s(\varphi_l^{r_ln_k}(i))}^{(s)}x_{\varphi^{r_ln_k}_l(i)}e_{\psi_{s}^{r_sn_k}(\varphi_{l}^{r_l n_k}(i))}\right\|
< \frac{1}{2k}.
\end{eqnarray}
By the second inequality in \eqref{a5},
\begin{eqnarray}\label{a15}
0 < \frac{1}{|x_i|} \leq 2 \;\; \mbox{ for } i\in I_k.
\end{eqnarray}
Now by \eqref{a12} and \eqref{a15} we get that for each $i\in I_k$ and $1\leq l \leq N,$
\begin{eqnarray*}
\left\| \left(\prod \limits_{v=1}^{r_ln_k} b_{\psi^{v}_l(i)}^{(l)}\right)e_{\psi^{r_ln_k}_l(i)}\right\|
&=&\left\|\frac{1}{x_i}  \left(\prod \limits_{v=1}^{r_ln_k} b_{\psi^{v}_l(i)}^{(l)}\right)x_{i}e_{\psi^{r_ln_k}_l(i)}\right\|\\
&\leq& 2 \left\| \left(\prod \limits_{v=1}^{r_ln_k} b_{\psi^{v}_l(i)}^{(l)}\right)x_{i}e_{\psi^{r_ln_k}_l(i)}\right\|\\
&<& \frac{1}{k}.
 \end{eqnarray*}
 Similarly, \eqref{a9}, \eqref{a13} and \eqref{a14} give that for each $i\in I_k$ and $1\leq s < l \leq N,$
 \begin{eqnarray*}
 &\;& \left\|\left( \prod\limits_{v = 0}^{r_s n_k -1} b_{\varphi^v_s (i)} ^{(s)}\right)^{-1}
   \left( \prod\limits_{v = 1}^{r_l n_k} b_{\psi^v_l (\varphi^{r_s n_k}_s(i))} ^{(l)}\right) e_{\psi_{l}^{r_ln_k}(\varphi_{s}^{r_s n_k}(i))}\right\|\nonumber\\
  &\;& =\left\|\left( \prod\limits_{v = 0}^{r_s n_k -1} b_{\varphi^v_s (i)} ^{(s)}x_{\varphi^{r_sn_k}_s(i)}\right)^{-1} \prod\limits_{v = 1}^{r_l n_k} b_{\psi^v_l (\varphi^{r_s n_k}_s(i))} ^{(l)}x_{\varphi^{r_sn_k}_s(i)} e_{\psi_{l}^{r_ln_k}(\varphi_{s}^{r_s n_k}(i))}\right\|\nonumber\\
    &\;&\leq 2 \left\|\prod\limits_{v = 1}^{r_l n_k} b_{\psi^v_l (\varphi^{r_s n_k}_s(i))} ^{(l)}x_{\varphi^{r_sn_k}_s(i)} e_{\psi_{l}^{r_ln_k}(\varphi_{s}^{r_s n_k}(i))}\right\|< \frac{1}{k}
  \end{eqnarray*}
  and
  \begin{eqnarray*}
   &\;&\left\| \left( \prod\limits_{v = 0}^{r_l n_k -1} b_{\varphi^v_l (i)} ^{(l)}\right)^{-1}
   \left( \prod\limits_{v = 1}^{r_s n_k} b_{\psi^v_s (\varphi^{r_l n_k}_l(i))} ^{(s)}\right) e_{\psi_{s}^{r_sn_k}(\varphi_{l}^{r_l n_k}(i))}\right\|\\
    &\;& = \left\|\left( \prod\limits_{v = 0}^{r_l n_k -1} b_{\varphi^v_l (i)} ^{(l)}x_{\varphi^{r_ln_k}_l(i)}\right)^{-1}\left(\prod\limits_{v = 1}^{r_s n_k} b_{\psi^v_s (\varphi^{r_l n_k}_l(i))} ^{(s)}\right) x_{\varphi^{r_ln_k}_l(i)} e_{\psi_{s}^{r_sn_k}(\varphi_{l}^{r_l n_k}(i))}\right\|\\
     &\;&<\frac{1}{k}.
    \end{eqnarray*}

     Taking $k=1, 2, 3, \ldots$ in the above argument, we can define inductively an increasing sequence $(n_k)_{k\geq 1}$ of positive integers by letting $n_k$ be a positive integer satisfying \eqref{a1} and \eqref{a2} for $N_0 = n_{k-1}$ (where we set $N_0 = 0$ when $k = 1$). Since for any fixed $i\in I,$ there exists an integer $m_0 \in \mathbb{N}$ such that $i\in I_k$ for all $k \geq m_0,$ the sequence
    $(n_k)_{k\geq 1}$ satisfies \textit{(H1)} and \textit{(H2)}.

  $(2)\Rightarrow (3).$ Suppose (2) holds and let $(n_k)_{k \geq 1}$ be an increasing sequence of positive integers satisfying \textit{(H1)} and \textit{(H2)}.
  Set $X_0 = X_1 = \cdots = X_N = span \{ e_i : i\in I\}$ which are dense in $X.$ For each $1 \leq l \leq N$ and integer $n\geq 1,$ we consider the linear mapping $S_{l,n} : X_l \rightarrow X$ given by
$$ S_{l, n}(e_i) = \left(\prod \limits_{v=0}^{r_ln_k-1} b_{\varphi^{v}_l(i)}^{(l)}\right)^{-1}e_{\varphi^{r_ln}_l(i)}\;\;\;\;\;(i\in I).$$
Since
 $T_l^{r_ln} e_i = \left(\prod \limits_{v=1}^{r_ln} b_{\psi^{v}_l(i)}^{(l)}\right)e_{\psi^{r_ln}_l(i)}\;(n\geq 1),$
 we have $T_l^{r_ln} S_{l,n} (e_i) = e_i$ for $i\in I$ and $n\geq 1.$ \\
  By \textit{(H1)}, for any $i \in I$ and $1 \leq l \leq N,$
 $$\lim \limits_{k\rightarrow \infty}S_{l, n_k}(e_i) = \lim \limits_{k\rightarrow \infty}\left(\prod \limits_{v=0}^{r_ln_k-1} b_{\varphi^{v}_l(i)}^{(l)}\right)^{-1}e_{\varphi^{r_ln_k}_l(i)} = 0$$
 and
 $$\lim \limits_{k\rightarrow \infty}T_l^{r_ln_k} e_i = \lim \limits_{k\rightarrow \infty}\left(\prod \limits_{v=1}^{r_ln_k} b_{\psi^{v}_l(i)}^{(l)}\right)e_{\psi^{r_ln_k}_l(i)} = 0.$$
 An easy calculation gives that for any $i \in I$ and $1 \leq s < l \leq N,$
  \begin{eqnarray*}
   \left\{\begin{array}{ll}
   T_l^{r_ln_k}S_{s, n_k}(e_i)=\left( \prod\limits_{v = 0}^{r_s n_k -1} b_{\varphi^v_s(i)} ^{(s)}\right)^{-1}
   \left( \prod\limits_{v = 1}^{r_l n_k} b_{\psi^v_l (\varphi^{r_s n_k}_s(i))}^{(l)} \right) e_{\psi_{l}^{r_ln_k}(\varphi_{s}^{r_s n_k}(i))},\\
   T_s^{r_sn_k}S_{l, n_k}(e_i)= \left( \prod\limits_{v = 0}^{r_l n_k -1} b_{\varphi^v_l(i)} ^{(l)}\right)^{-1}
   \left( \prod\limits_{v = 1}^{r_s n_k} b_{\psi^v_s (\varphi^{r_l n_k}_l(i))} ^{(s)}\right) e_{\psi_{s}^{r_sn_k}(\varphi_{l}^{r_l n_k}(i))},\\
   \end{array}\right.
   \end{eqnarray*}
   therefore by \textit{(H2)}, $T_l^{r_ln_k}S_{s, n_k}(e_i)$ and $T_s^{r_sn_k}S_{l, n_k}(e_i)$ both tend to $0$ for any $i\in I.$ Then we conclude by using linearity.

$(3)\Rightarrow (1).$ This implication is obvious.
\end{proof}
\begin{remark}
For each $1 \leq l \leq N$, the mapping $\varphi_{l} : I\rightarrow I$ has no periodic points does not implies $(\varphi_{s}^{r_sn})_n$ is run-away with $(\varphi_{l}^{r_ln})_n$ for $1 \leq s < l \leq N.$ For example, if we let $I = \mathbb{Z},$ define $\varphi_1(i) = i+2$ and $\varphi_2(i) = i+1$ for $i\in \mathbb{Z}.$ Let $r_1 = 1, r_2 = 2.$ Clearly, for $l =1,2,$ $\varphi_{l}$ has no periodic points. But for any positive integer $n\geq 1,$ $\varphi_1^{r_1 n} = \varphi_2^{r_2 n}.$
\end{remark}

Now we illustrate Theorem \ref{main 1} with the following example. The definitions about shifts on weighted $L^p$ spaces of directed trees are borrowed from Mart\'{\i}nez-Avenda\~{n}o \cite{MRA}.

\begin{example}
Let $T = (V,E)$ be a infinite unrooted directed tree such that $T$ have no vertices of outdegree larger than $1.$ Let $1\leq p < \infty,$ and $\lambda = \{\lambda_v\}_{v\in V}$ be a sequence of positive numbers such that $\sup\limits_{u\in V,v=chi(u)}\frac{\lambda_{v}}{\lambda_u} < \infty$ and $\sup\limits_{u\in V,v=chi^2(u)}\frac{\lambda_{v}}{\lambda_u} < \infty$. We denote by $L^p(T, \lambda)$ the space of complex-valued functions $f: V\rightarrow \mathbb{C}$ such that
$$\sum\limits_{v\in V}|f(v)|^p\lambda_{v} < \infty.$$
$L^p(T, \lambda)$ is a Banach space with the norm $$\|f\|_p = \left(\sum\limits_{v\in V}|f(v)|^p\lambda_{v} \right)^{\frac{1}{p}}.$$
Now we define the shifts $S_1$ and $S_2$ on $L^p(T, \lambda)$ by
$$(S_1f)(v) = f(par(v))\;\; \mbox { for } f\in L^p(T, \lambda)$$
and
$$(S_2f)(v) = f(par^2(v))\;\; \mbox { for } f\in L^p(T, \lambda).$$
The weight assumption ensure that $S_1$ and $S_2$ are bounded on $L^p(T, \lambda).$
Let $f$ be any vector in $L^p(T, \lambda),$ if we identify $f$ by $(f(v))_{v\in V},$ then $L^p(T, \lambda)$ can be seen as a Banach sequence space over $I : = V.$ Let $u\in V,$ and denote by $\chi_u$ the characteristic function of vertex $u.$ Define $e_u: = \chi_u.$ Clearly, $(e_u)_{u\in V}$ is an OP-basis of $L^p(T, \lambda).$
In this interpretation, for $l=1,2,$ $S_l$ is a weighted pseudo-shift $T_{b^{(l)}, \varphi_l}$ with
\begin{eqnarray*}
b^{(l)}_v = 1 \mbox{ and } \varphi_l(v) = par^l(v)\;\;\;( v\in V).
\end{eqnarray*}
Thus for $l=1,2,$ the inverse $\psi_l = \varphi_l^{-1} : par^l(V)\rightarrow V$ is given by
\begin{eqnarray*}
 \psi_l(v) = chi^l(v)\;\;\;\mbox{ for } v\in par^l(V).
\end{eqnarray*}
If we set
\begin{eqnarray*}
b^{(l)}_{chi(v)} = 0 \mbox{ and }\lambda_{chi(v)}=0 \;\;\mbox{ if } v\in V\backslash par(V),
\end{eqnarray*}
then by Theorem \ref{main 1}, $S_1, S_2^2$ are densely d-hypercyclic if and only if there exists an increasing sequence $(n_k)_{k\geq 1}$ of positive integers such that for every $v\in V$ and $l=1,2,$

   \begin{eqnarray*}
   \lim \limits_{k\rightarrow \infty}\left\|\left( \prod\limits_{t = 0}^{l n_k -1} b_{\varphi^t_{l}(v)} ^{(l)}\right)^{-1} e_{\varphi^{l n_k}_{l}(v)}\right\|=\lim \limits_{k\rightarrow \infty}\left\|\chi_{par^{l^2n_k}(v)} \right\|=\lim \limits_{k\rightarrow \infty} \left(\lambda_{par^{l^2n_k}(v)}\right)^{\frac{1}{p}}=0,
   \end{eqnarray*}

    \begin{eqnarray*}
    \lim \limits_{k\rightarrow \infty}\left\|\left( \prod\limits_{t = 1}^{l n_k} b_{\psi^t_{l}(v)} ^{(l)}\right) e_{\psi^{l n_k}_{l}(v)}\right\|=\lim \limits_{k\rightarrow \infty}\left\|\chi_{chi^{l^2n_k}(v)} \right\|=\lim \limits_{k\rightarrow \infty} \left(\lambda_{chi^{l^2n_k}(v)}\right)^{\frac{1}{p}}=0
   \end{eqnarray*}
 and for $s=1, l=2,$
\begin{eqnarray*}
   &\;&\lim \limits_{k\rightarrow \infty}\left\|\left( \prod\limits_{t = 0}^{s n_k -1} b_{\varphi^t_{s}(v)} ^{(s)}\right)^{-1}
   \left( \prod\limits_{t = 1}^{l n_k} b_{\psi^t_{l} (\varphi^{s n_k}_{s}(v))}^{(l)} \right) e_{\psi_{l}^{ln_k}(\varphi_{s}^{s n_k}(v))} \right\|\\
    &\;& = \lim \limits_{k\rightarrow \infty}\left\|\chi_{chi^{(l^2-s^2)n_k}(v)} \right\|
    =\lim \limits_{k\rightarrow \infty} \left(\lambda_{chi^{3n_k}(v)}\right)^{\frac{1}{p}}=0,
   \end{eqnarray*}

   \begin{eqnarray*}
    &\;&\lim \limits_{k\rightarrow \infty}\left\|\left( \prod\limits_{t = 0}^{l n_k -1} b_{\varphi^t_{l}(v)} ^{(l)}\right)^{-1}
   \left( \prod\limits_{t = 1}^{s n_k} b_{\psi^t_{s} (\varphi^{l n_k}_{l}(v))} ^{(s)}\right) e_{\psi_{s}^{sn_k}(\varphi_{l}^{l n_k}(v))}\right\|\\
   &\;& = \lim \limits_{k\rightarrow \infty}\left\|\chi_{par^{(l^2-s^2)n_k}(v)} \right\|
    =\lim \limits_{k\rightarrow \infty} \left(\lambda_{par^{3n_k}(v)}\right)^{\frac{1}{p}}=0.
   \end{eqnarray*}

   If we let $s_1, s_2\in \mathbb{R}$ with $1 < s_1 \leq s_2.$  Select an arbitrary fixed vertex and call it $\omega^*.$ For each $u\in V,$ set
 \begin{eqnarray*}
   \lambda_u = \left\{\begin{array}{ll}
  \frac{1}{s_1^d} \;\; \mbox{ if } \omega^* \mbox{ is a descendant of } u,\\
 \\
   \frac{1}{s_2^d} \;\; \mbox{ if } \omega^*  \mbox{ is an ancestor of } u,\\
   \end{array}\right.
   \end{eqnarray*}
where $d = \mbox{ dist } (u, \omega^*).$\\
Fix $u\in V$ with $d = \mbox{ dist } (u, \omega^*)$ . Also, let $n > d,$ we then have that
  \begin{eqnarray*}
   \lambda_{par^{n}(u)}= (\frac{1}{s_1})^{n\pm d},
  \end{eqnarray*}
  where the plus sign corresponds to the case where $\omega^*$ is a descendant of $u$ and the minus sign corresponds to the case where $\omega^*$ is an ancestor of $u.$
  And
  \begin{eqnarray*}
  \lambda_{chi^{n}(u)} = \left\{\begin{array}{ll}
   \frac{1}{s_2^{n - d}} \;\; \mbox{ if } \omega^* \mbox{ is a descendant of } u \mbox{ and } u\in par^n(V),\\
 \\
   \frac{1}{s_2^{n + d}} \;\; \mbox{ if } \omega^*  \mbox{ is an ancestor of } u \mbox{ and } u\in par^n(V),\\
   \\
   0 \;\; \;\;\;\;\;\;\mbox{ if } u\in V\backslash par^n(V).
   \end{array}\right.
  \end{eqnarray*}
  Therefore $\lim \limits_{n\rightarrow \infty}\lambda_{par^{l^2n}(u)} = \lim \limits_{n\rightarrow \infty}\lambda_{chi^{l^2n}(u)}=0$ for $l=1,2$ and $\lim \limits_{n\rightarrow \infty}\lambda_{par^{3n}(u)} = \lim \limits_{n\rightarrow \infty}\lambda_{chi^{3n}(u)}=0.$
\end{example}

\section{Disjoint supercyclic powers of weighted pseudo-shifts}
 In this section, we will extend the characterizations in Theorem \ref{main 1} from d-hypercyclicity to d-supercyclicity and will present some corollaries.

\begin{theorem}\label{Pd-superbi1}
Let $X$ be a Banach sequence space over $I$, in which $(e_i)_{i\in I}$ is an OP-basis. Let integers $1\leq r_1 < r_2 < \cdots < r_{N}$ be given. $N\geq 2,$ for each $1 \leq l \leq N$, let $T_{l} = T_{b^{(l)}, \varphi_l} : X \rightarrow X$ be a weighted pseudo-shift with weight sequence $b^{(l)} = (b_i^{(l)})_{i\in I}$ and injective mapping $\varphi_{l}.$ If for each $1 \leq l \leq N,$ $(\varphi^n_l)_n$ is a run-away sequence and for $1 \leq s < l \leq N,$ the sequence $(\varphi_{s}^{r_sn})_n$ is run-away with $(\varphi_{l}^{r_ln})_n$. Then the following statements are equivalent:

\begin{enumerate}
\item $T_{1} ^{r_1}, T_{2} ^{r_2}, \ldots, T_{N} ^{r_N}$ have a dense set of d-supercyclic vectors.

\item There exists an increasing sequence $(n_k)_{k\geq 1}$ of positive integers such that:

    (H1)For any $i, j\in I$ and $1 \leq l, s\leq N,$

   \begin{eqnarray*}
   \lim \limits_{k\rightarrow \infty} \left\|\left( \prod\limits_{v = 0}^{r_l n_k -1} b_{\varphi^v_l(i)} ^{(l)}\right)^{-1} e_{\varphi_l^{r_l n_k}(i)} \right\|
   \left\|\left( \prod\limits_{v = 1}^{r_s n_k} b_{\psi^v_s(j)} ^{(s)}\right) e_{\psi^{r_s n_k}_s(j)}\right\|=0.
   \end{eqnarray*}

(H2)For every $i\in I$ and  $1 \leq s < l \leq N,$
\begin{eqnarray*}
   \left\{\begin{array}{ll}
    (i)\;\lim \limits_{k\rightarrow \infty} \left\|\left( \prod\limits_{v = 0}^{r_s n_k -1} b_{\varphi^v_s(i)} ^{(s)}\right)^{-1}
   \left( \prod\limits_{v = 1}^{r_l n_k} b_{\psi^v_l (\varphi^{r_s n_k}_s(i))}^{(l)} \right) e_{\psi_{l}^{r_ln_k}(\varphi_{s}^{r_s n_k}(i))} \right\|= 0,\\
   \\
     (ii)\;\lim \limits_{k\rightarrow \infty} \left\|\left( \prod\limits_{v = 0}^{r_l n_k -1} b_{\varphi^v_l(i)} ^{(l)}\right)^{-1}
   \left( \prod\limits_{v = 1}^{r_s n_k} b_{\psi^v_s (\varphi^{r_l n_k}_l(i))} ^{(s)}\right) e_{\psi_{s}^{r_sn_k}(\varphi_{l}^{r_l n_k}(i))}\right\|=0.\\
   \end{array}\right.
   \end{eqnarray*}

   \item  $T_1^{r_1 },  T_{2} ^{r_2}, \ldots,T_{N }^{r_N }$ satisfy the d-Supercyclicity Criterion.
   \end{enumerate}
\end{theorem}

\begin{proof}

$(1)\Rightarrow (2).$
 Fix
 $$I : =\{i_1, i_2, \cdots, i_n,\cdots \}$$ and for each $k\in \mathbb{N}$ with $k\geq 1$ set $I_k: = \{i_1, i_2, \cdots, i_k\}.$
To prove $(2)$, it is enough to verify that for any positive integer $k\geq 1$ and any $N_0 \in \mathbb{\mathbb{N}},$ there is an integer $n_k > N_0$ such that:

For any $i, j\in I_k$ and $1 \leq l, s\leq N,$

   \begin{eqnarray*}
   \left\|\left( \prod\limits_{v = 0}^{r_l n_k -1} b_{\varphi^v_l(i)} ^{(l)}\right)^{-1} e_{\varphi^{r_l n_k}_l(i)} \right\|
   \left\|\left( \prod\limits_{v = 1}^{r_s n_k} b_{\psi^v_s(j)} ^{(s)}\right) e_{\psi^{r_s n_k}_s(j)}\right\|< \frac{1}{k}.
   \end{eqnarray*}

 For every $i\in I_k$ and $1 \leq s < l \leq N,$

\begin{eqnarray*}
   \left\{\begin{array}{ll}
   (i)\;\left\|\left( \prod\limits_{v = 0}^{r_s n_k -1} b_{\varphi^v_s (i)} ^{(s)}\right)^{-1}
   \left( \prod\limits_{v = 1}^{r_l n_k} b_{\psi^v_l (\varphi^{r_s n_k}_s(i))} ^{(l)}\right) e_{\psi_{l}^{r_ln_k}(\varphi_{s}^{r_s n_k}(i))}\right\| < \frac{1}{k}, \\
   \\
  (ii)\; \; \left\| \left( \prod\limits_{v = 0}^{r_l n_k -1} b_{\varphi^v_l (i)} ^{(l)}\right)^{-1}
   \left( \prod\limits_{v = 1}^{r_s n_k} b_{\psi^v_s (\varphi^{r_l n_k}_l(i))} ^{(s)}\right) e_{\psi_{s}^{r_sn_k}(\varphi_{l}^{r_l n_k}(i))}\right\| < \frac{1}{k}. \\
   \end{array}\right.
   \end{eqnarray*}
Let integers $k\geq 1$ and $N_0 \in \mathbb{N}$ be given. By assumption there is some $\delta_k > 0$ such that for any $x = (x_i)_{i\in I} \in X,$
   \begin{eqnarray*}
   ||x_i e_i|| < \frac{1}{2k} \;\; \mbox{ for } i\in I, \mbox{ if } ||x|| < \delta_k.
    \end{eqnarray*}
  Let $\widetilde{N_0}\in \mathbb{N}$ be the integer such that
\begin{eqnarray*}
 \left\{\begin{array}{ll}
(i)\;\varphi^{r_ln}_l(I_k)\cap I_k = \emptyset \; (1\leq l \leq N),\\
\\
(ii)\;\psi_{l}^{r_ln}(\varphi_{s}^{r_s n}(I_k)\cap  \varphi^{r_ln}_l(I))\cap I_k = \emptyset \; (1\leq s < l \leq N)
\end{array}\right.
\end{eqnarray*}
for all $n\geq \widetilde{N_0}.$

    Since the d-supercyclic vectors are dense in $X,$ there exists a d-suppercyclic vector $x = (x_i)_{i\in I}\in X$ such that
      \begin{eqnarray*}
\left\|x-\sum\limits_{i\in I_k}e_i\right\| < \delta_k.
   \end{eqnarray*}
    Now, let $A = \{\alpha (T_1^{r_1n}x, T_2^{r_2n}x,\ldots, T_N^{r_Nn}x) : \alpha \in \mathbb{C}, n\in \mathbb{N}\}.$ Obviously, $A$ is dense in $X^N.$ For every $p\in \mathbb{N},$ let $B_p = \{\alpha (T_1^{r_1n}x, T_2^{r_2n}x,\ldots, T_N^{r_Nn}x) : \alpha \in \mathbb{C}, n\in \mathbb{N}, n \leq p\}.$ Since $X$ is an infinite dimensional Banach space, $A\setminus B_p$ remains dense in $X^N$ for any $p\in \mathbb{N}.$ Thus we can find a complex number $\alpha \neq 0$
  and an integer $n_k > \mbox{ max } \{N_0, \widetilde{N_0}\}$ such that for each $1\leq l \leq N,$
      \begin{eqnarray*}
       \left\|\alpha y^{(l)} -\sum_{i\in I_k}e_i\right\| < \delta_k,
     \end{eqnarray*}
  where $y^{(l)} := T_l^{r_ln_k}x = (\prod\limits_{v=0}^{r_ln_k-1}b_{\varphi^v_l(i)}^{(l)}x_{\varphi^{r_ln_k}_l(i)})_{i\in I}=(y^{(l)}_i)_{i\in I}.$\\
  By the continuous inclusion of $X$ into $\mathbb{K}^I,$  we can in addition assume that
     \begin{eqnarray*}
      \left\{\begin{array}{ll}
     (i) \; \sup\limits_{i\in I_k}|x_i-1|\leq \frac{1}{2}\\
     \;\\
     (ii)\sup\limits_{i\in I_k}|\alpha y_i^{(l)}-1|\leq \frac{1}{2} \;\;\mbox{ for }1\leq l \leq N.\\
     \end{array}\right.
   \end{eqnarray*}
  It follows that for any $i\in I_k,$
\begin{eqnarray}\label{3.8}
\left\{\begin{array}{ll}
(i)\; 0 < \frac{1}{|x_i|} \leq 2, \\
\\
(ii)\;x_{\varphi^{r_ln_k}_l(i)} \neq 0\;\; (1\leq l \leq N),\\
\\
(iii)\;0<\left|\left(\alpha\prod\limits_{v=0}^{r_ln_k-1}b_{\varphi^v_l(i)}^{(l)}x_{\varphi^{r_ln_k}_l(i)}\right)^{-1}\right| \leq 2\;\; (1\leq l \leq N).\\
\end{array}\right.
\end{eqnarray}
Repeating a similar argument as in the proof of Theorem \ref{main 1}, one can obtain that:

For each $i\in I_k$ and $1\leq l \leq N,$
\begin{eqnarray}\label{3.9}
\left\| \left(\alpha\prod \limits_{v=0}^{r_ln_k-1} b_{\varphi^{v}_l(i)}^{(l)}\right)^{-1}e_{\varphi^{r_ln_k}_l(i)}\right\|
 <  \frac{1}{k}
\end{eqnarray}
and
\begin{eqnarray}\label{3.10}
\left\| \left(\alpha\prod \limits_{v=1}^{r_ln_k} b_{\psi^{v}_l(i)}^{(l)}\right)e_{\psi^{r_ln_k}_l(i)}\right\|<\frac{1}{k}.
 \end{eqnarray}

 For each $i\in I_k$ and $1\leq s < l \leq N,$
\begin{eqnarray}\label{3.11}
\left\|\alpha\prod\limits_{v=1}^{r_ln_k}b_{\psi^v_l(\varphi_s^{r_sn}(i))}^{(l)}x_{\varphi^{r_sn_k}_s(i)}e_{\psi_{l}^{r_ln_k}(\varphi_{s}^{r_s n_k}(i))}\right\|
< \frac{1}{2k}
\end{eqnarray}
and
\begin{eqnarray}\label{3.12}
\left\|\alpha\prod\limits_{v=1}^{r_sn_k}b_{\psi^v_s(\varphi_l^{r_ln_k}(i))}^{(s)}x_{\varphi^{r_ln_k}_l(i)}e_{\psi_{s}^{r_sn_k}(\varphi_{l}^{r_l n_k}(i))}\right\|
< \frac{1}{2k}.
\end{eqnarray}

 Hence by \eqref{3.9} and \eqref{3.10} for any $i, j\in I_k$ and $1 \leq l, s\leq N,$

   \begin{eqnarray*}
   &\;&\left\|\left( \prod\limits_{v = 0}^{r_l n_k -1} b_{\varphi^v_l(i)} ^{(l)}\right)^{-1} e_{\varphi^{r_l n_k}_l(i)} \right\|
   \left\|\left( \prod\limits_{v = 1}^{r_s n_k} b_{\psi^v_s(j)} ^{(s)}\right) e_{\psi^{r_s n_k}_s(j)}\right\|\\
   &\;& =\left\|\left(\alpha \prod\limits_{v = 0}^{r_l n_k -1} b_{\varphi^v_l(i)} ^{(l)}\right)^{-1} e_{\varphi^{r_l n_k}_l(i)} \right\|
   \left\|\left(\alpha \prod\limits_{v = 1}^{r_s n_k} b_{\psi^v_s(j)} ^{(s)}\right) e_{\psi^{r_s n_k}_s(j)}\right\|
   <\frac{1}{k^2} \leq \frac{1}{k}.
   \end{eqnarray*}

 By \eqref{3.8}, \eqref{3.11} and \eqref{3.12} for each $i\in I_k$ and $1\leq s < l \leq N,$
 \begin{eqnarray}
 &\;& \left\|\left( \prod\limits_{v = 0}^{r_s n_k -1} b_{\varphi^v_s (i)} ^{(s)}\right)^{-1}
   \left( \prod\limits_{v = 1}^{r_l n_k} b_{\psi^v_l (\varphi^{r_s n_k}_s(i))} ^{(l)}\right) e_{\psi_{l}^{r_ln_k}(\varphi_{s}^{r_s n_k}(i))}\right\|\nonumber\\
  &\;& =\left\|\left(\alpha \prod\limits_{v = 0}^{r_s n_k -1} b_{\varphi^v_s (i)} ^{(s)}x_{\varphi^{r_s n_k}_s(i)}\right)^{-1} \alpha \prod\limits_{v = 1}^{r_l n_k} b_{\psi^v_l (\varphi^{r_s n_k}_s(i))} ^{(l)}x_{\varphi^{r_sn_k}_s(i)} e_{\psi_{l}^{r_ln_k}(\varphi_{s}^{r_s n_k}(i))}\right\|\nonumber\\
    &\;&\leq 2 \left\|\alpha\prod\limits_{v = 1}^{r_l n_k} b_{\psi^v_l (\varphi^{r_s n_k}_s(i))} ^{(l)}x_{\varphi^{r_sn_k}_s(i)} e_{\psi_{l}^{r_ln_k}(\varphi_{s}^{r_s n_k}(i))}\right\|< \frac{1}{k}
  \end{eqnarray}
  and
  \begin{eqnarray*}
   &\;&\left\| \left( \prod\limits_{v = 0}^{r_l n_k -1} b_{\varphi^v_l (i)} ^{(l)}\right)^{-1}
   \left( \prod\limits_{v = 1}^{r_s n_k} b_{\psi^v_s (\varphi^{r_l n_k}_l(i))} ^{(s)}\right) e_{ \psi_{s}^{r_sn_k}(\varphi_{l}^{r_l n_k}(i))}\right\|\\
    &\;& = \left\|\left(\alpha \prod\limits_{v = 0}^{r_l n_k -1} b_{\varphi^v_l (i)} ^{(l)}x_{\varphi^{r_ln_k}_l(i)}\right)^{-1}\left(\alpha \prod\limits_{v = 1}^{r_s n_k} b_{\psi^v_s (\varphi^{r_l n_k}_l(i))} ^{(s)}\right) x_{\varphi^{r_ln_k}_l(i)} e_{\psi_{s}^{r_sn_k}(\varphi_{l}^{r_l n_k}(i))}\right\|\\
    &\;& <\frac{1}{k}.
    \end{eqnarray*}
   Thus we complete the proof.

  $(2)\Rightarrow (3).$ Suppose (2) holds and let $(n_k)_{k \geq 1}$ be an increasing sequence of positive integers satisfying \textit{(H1)} and \textit{(H2)}.
  Let's show that $T_{1} ^{r_1}, T_{2} ^{r_2}, \ldots, T_{N} ^{r_N}$ satisfy the d-Supercyclicity Criterion. Set $X_0 = X_1 = \cdots = X_N = span \{ e_i : i\in I\}.$ For each $1 \leq l \leq N$ and $n\in \mathbb{N}$ with $n\geq 1,$ we consider the linear mappings: $S_{l,n} : X_l \rightarrow X$ given by
$$ S_{l, n}(e_i) = \left(\prod \limits_{v=0}^{r_ln-1} b_{\varphi^{v}_l(i)}^{(l)}\right)^{-1}e_{\varphi^{r_ln}_l(i)}\;\;\;\;\;(i\in I).$$
 The same argument as used in the proof of $(2)\Rightarrow (3)$ in Theorem \ref{main 1} yields that $(i)$ of Definition \ref{d-supercriterion} is satisfied. So we just need to check that $T_{1} ^{r_1}, T_{2} ^{r_2}, \ldots, T_{N} ^{r_N}$ satisfy condition $(ii)$ in Definition \ref{d-supercriterion} with respect to $(n_k)_{k \geq 1}$. Let $y_0, y_1, \ldots, y_N \in span \{ e_i : i\in I\},$ there exists an $k_0 \in \mathbb{N}$ such that
   \begin{eqnarray*}
   y_i = \sum\limits_{j\in I_{k_0}} y_{i,j}e_j \;\; (0 \leq i \leq N).
   \end{eqnarray*}
   Set $C : = \max \{|y_{i,j}| : 0 \leq i \leq N, j\in I_{k_0}\}.$
   By \textit{(H1)}, for any $i,j \in I$ and $1 \leq l, s \leq N,$
  \begin{eqnarray*}
  &\;&\lim \limits_{k\rightarrow \infty}\|T_l^{r_l n_k} e_i\|\|S_{s, n_k}e_j\|\\
  &\;&= \lim \limits_{k\rightarrow \infty}\left\|\left(\prod \limits_{v=1}^{r_ln_k} b_{\psi^{v}_l(i)}^{(l)}\right)e_{\psi^{r_ln_k}_l(i)}\right\|
  \left\|\left(\prod \limits_{v=0}^{r_s n_k-1} b_{\varphi^{v}_s(j)}^{(s)}\right)^{-1}e_{\varphi^{r_s n_k}_s(j)}\right\|=0.
  \end{eqnarray*}
  It follows that
  \begin{eqnarray*}
  &\;&\left\|T_l^{r_l n_k} y_0\right\| \left\|\sum \limits_{s=1}^{N}S_{s, n_k}y_s\right\|\\
  &=&\left\|\sum\limits_{j\in I_{k_0}}y_{0,j}\prod \limits_{v=1}^{r_ln_k} b_{\psi^{v}_l(j)}^{(l)}e_{\psi^{r_ln_k}_l(j)}\right\|
  \left\|\sum \limits_{s=1}^{N}\sum\limits_{j\in I_{k_0}}y_{s,j}\left(\prod \limits_{v=0}^{r_s n_k-1} b_{\varphi^{v}_s(j)}^{(s)}\right)^{-1}e_{\varphi^{r_s n_k}_s(j)}\right\|\\
   &\leq&C^2 \left(\sum\limits_{j\in I_{k_0}}\left\|\prod \limits_{v=1}^{r_ln_k} b_{\psi^{v}_l(j)}^{(l)}e_{\psi^{r_ln_k}_l(j)}\right\|\right) \left(\sum \limits_{s=1}^{N}\sum\limits_{j\in I_{k_0}}\left\|\left(\prod \limits_{v=0}^{r_s n_k-1} b_{\varphi^{v}_s(j)}^{(s)}\right)^{-1}e_{\varphi^{r_s n_k}_s(j)}\right\|\right)\\
   \\
   & \overrightarrow{k\rightarrow\infty}& 0.
  \end{eqnarray*}

$(3)\Rightarrow (1).$ This implication is obvious.
\end{proof}

Next, we consider the special case $\varphi_1 = \varphi_2 = \cdots\cdots = \varphi_N$ in Theorem \ref{Pd-superbi1}.
\begin{corollary}\label{Pd-superbi}
Let $X$ be a Banach sequence space over $I$, in which $(e_i)_{i\in I}$ is an OP-basis. $N\geq 2,$ for each $1 \leq l \leq N$, let $T_{l} = T_{b^{(l)}, \varphi} : X \rightarrow X$ be a weighted pseudo-shift with weight sequence $b^{(l)} = (b_i^{(l)})_{i\in I}.$   Then for any integers $1\leq r_1 < r_2 < \cdots\cdots < r_{N},$ the following are equivalent:

\begin{enumerate}
\item $T_{1} ^{r_1}, T_{2} ^{r_2}, \ldots, T_{N} ^{r_N}$ have a dense set of d-supercyclic vectors.

\item $(\alpha)$ The mapping $\varphi : I\rightarrow I$ has no periodic points.

     \quad  \; $(\beta)$ There exists an increasing sequence $(n_k)_{k\geq 1}$ of positive integers such that:

    (H1)For any $i, j\in I$ and $1 \leq l, s\leq N$ we have

   \begin{eqnarray*}
   \lim \limits_{k\rightarrow \infty} \left\|\left( \prod\limits_{v = 0}^{r_l n_k -1} b_{\varphi^v(i)} ^{(l)}\right)^{-1} e_{\varphi^{r_l n_k}(i)} \right\|
   \left\|\left( \prod\limits_{v = 1}^{r_s n_k} b_{\psi^v(j)} ^{(s)}\right) e_{\psi^{r_s n_k}(j)}\right\|=0.
   \end{eqnarray*}

(H2)For every $i\in I$ and any  $1 \leq s < l \leq N,$
\begin{eqnarray*}
   \left\{\begin{array}{ll}
    (i)\;\lim \limits_{k\rightarrow \infty} \left\|\left( \prod\limits_{v = 0}^{r_s n_k -1} b_{\varphi^v(i)} ^{(s)}\right)^{-1}
   \left( \prod\limits_{v = 1}^{r_l n_k} b_{\psi^v (\varphi^{r_s n_k}(i))}^{(l)} \right) e_{\psi^{(r_l-r_s) n_k}(i)} \right\|= 0,\\
   \\
     (ii)\;\lim \limits_{k\rightarrow \infty} \left\|\left( \prod\limits_{v = 0}^{r_l n_k -1} b_{\varphi^v(i)} ^{(l)}\right)^{-1}
   \left( \prod\limits_{v = 1}^{r_s n_k} b_{\psi^v (\varphi^{r_l n_k}(i))} ^{(s)}\right) e_{\varphi^{(r_l-r_s) n_k}(i)}\right\|=0.\\
   \end{array}\right.
   \end{eqnarray*}

   \item  $T_1^{r_1 },  T_{2} ^{r_2}, \ldots,T_{N }^{r_N }$ satisfy the d-Supercyclicity Criterion.
   \end{enumerate}
\end{corollary}

\begin{proof}
By Theorem \ref{Pd-superbi1}, we just need to prove that $(1)$ implies $(\alpha).$ Suppose on the contrary that $\varphi$ has a periodic point, then there exists an $i\in I$ and an integer $M\geq 1$ such that $\varphi ^M (i) = i.$ For each $1 \leq l \leq N$ and any $x\in X,$ the entry of $T_l^{r_ln} x$ at position $i$ is $\left(\prod \limits_{v=0}^{r_ln-1}b_{\varphi^v(i)}^{(l)}\right)x_{\varphi^{r_ln}(i)}.$ Since $\varphi ^M (i) = i,$ both $\left(b_{\varphi^v(i)}^{(l)}\right)_v$ and $(x_{\varphi^{r_ln}(i)})_n$ are periodic sequences. Which implies that  $\left\{\frac{\left(\prod \limits_{v=0}^{r_1n-1}b_{\varphi^v(i)}^{(1)}\right)x_{\varphi^{r_1n}(i)}}{\left(\prod \limits_{v=0}^{r_2n-1}b_{\varphi^v(i)}^{(2)}\right)x_{\varphi^{r_2n}(i)}}\right\}_{n\in\mathbb{N}}$ can not be dense in $\mathbb{K},$ it follows that the set $$\left\{\left(\alpha\prod \limits_{v=0}^{r_1n-1}b_{\varphi^v(i)}^{(1)} x_{\varphi^{r_1n}(i)}, \alpha \prod \limits_{v=0}^{r_2n-1}b_{\varphi^v(i)}^{(2)}x_{\varphi^{r_2n}(i)}\right) : \alpha \in \mathbb{C},n\in\mathbb{N}\right\}$$ can not be dense in $\mathbb{K}^2.$ By continuous inclusion of $X$ into $\mathbb{K}^I,$ the set $$\{\alpha\left(T_1^{r_1n} x, T_2^{r_2n} x, \ldots, T_N^{r_Nn} x \right): \alpha\in \mathbb{C}, n\in \mathbb{N}\}$$ can not be dense in $X^N,$ which is contrary to condition (1). Hence $\varphi$ has no periodic points.
\end{proof}

If the mapping $\varphi$ satisfy that each $i\in I$ lies outside $\varphi^n(I)$ for all sufficiently large n, which implies in particular that the sequence $(\varphi^n)_n$ is run-away. In this case, for every $i\in I,$ $ e_{\psi^n(i)} = 0$ is eventually 0 when $n$ is large enough by the definition of $\psi^n.$ Thus $(H1)$ and $(i)$ of $(H2)$ in Corollary \ref{Pd-superbi} is automatically satisfied. Therefore the following result holds.
\begin{corollary}\label{uni}
Let $X$ be a Banach sequence space over $I$, in which $(e_i)_{i\in I}$ is an OP-basis. Let integers $1\leq r_1 < r_2 < \cdots\cdots < r_{N}$ be given. For each $1 \leq l \leq N$, let $T_{l} = T_{b^{(l)},\varphi} : X \rightarrow X$ be a weighted pseudo-shift with weight sequence $b^{(l)} = (b_i^{(l)})_{i\in I},$ so that each $i\in I$ lies outside $\varphi^n(I)$ for all sufficiently large n. Then  the following assertions are equivalent:

\begin{enumerate}
\item $T_{1} ^{r_1}, T_{2} ^{r_2}, \ldots, T_{N} ^{r_N}$ are densely d-supercyclic.

\item There exists an increasing sequence $(n_k)_{k\geq 1}$ of positive integers such that for every $i\in I$ we have:

\begin{eqnarray*}
    \lim \limits_{k\rightarrow \infty} \left\| \left( \prod\limits_{v = 0}^{r_l n_k -1} b_{\varphi^v(i)} ^{(l)}\right)^{-1}
   \left( \prod\limits_{v = 1}^{r_s n_k} b_{\psi^v (\varphi^{r_l n_k}(i))} ^{(s)}\right) e_{\varphi^{(r_l-r_s) n_k}(i)}\right\|= 0\;\;(1 \leq s < l \leq N).
   \end{eqnarray*}

\item  $T_1^{r_1 }, T_{2} ^{r_2}, \ldots, T_{N }^{r_N }$ satisfy the d-Supercyclicity Criterion.
\end{enumerate}
\end{corollary}
\begin{example}
Let $X = \ell^2(\mathbb{N}),$ and integers $1\leq r_1 < r_2 < \cdots\cdots < r_{N}\;(N\geq 2)$ be given. For each $1\leq l \leq N,$ let $(a_{l,n})_{n=1}^\infty$ be a bounded sequence of nonzero scalars and $T_l$ be the unilateral backward weighted shift on $X$ defined by $$T_l e_0 =0 \mbox{ and } T_l e_j = a_{l,j}e_{j-1} \;\;\mbox{ for }j\geq 1,$$
where $(e_j)_{j\in \mathbb{N}}$ is the canonical basis of $\ell^2(\mathbb{N}).$ Clearly, in this case, $X$ is a Banach sequence space over $I = \mathbb{N}$ with OP-basis $(e_j)_{j\in \mathbb{N}}.$ Each $T_l$ is the pseudo-shift $T_{b^{(l)},\varphi}$ with
$$b_i^{(l)} = a_{l, i+1} \mbox{ and } \varphi (i) = i+1 \mbox{ for any } i\in \mathbb{N}.$$
 By Corollary \ref{uni}, $T_{1} ^{r_1}, T_{2} ^{r_2}, \ldots, T_{N} ^{r_N}$ are densely d-supercyclic if and only if
there exists an increasing sequence $(n_k)_{k\geq 1}$ of positive integers such that for every $i\in I$ and $1 \leq s < l \leq N,$
\begin{eqnarray*}
   &\;& \lim \limits_{k\rightarrow \infty} \left\| \left( \prod\limits_{v = 0}^{r_l n_k -1} b_{\varphi^v(i)} ^{(l)}\right)^{-1}
   \left( \prod\limits_{v = 1}^{r_s n_k} b_{\psi^v (\varphi^{r_l n_k}(i))} ^{(s)}\right) e_{\varphi^{(r_l-r_s) n_k}(i)}\right\| \\
   &\;&= \lim \limits_{k\rightarrow \infty}\left| \left( \prod\limits_{v = i+1}^{i+r_l n_k } a_{l,v}\right)^{-1}
   \left(\prod\limits_{v = i+(r_l-r_s)n_k+1}^{i+r_l n_k} a_{s,v}\right) \right| =0.
   \end{eqnarray*}
\end{example}

\begin{remark}
Recently, disjoint hypercyclic and disjoint supercyclic weighted translations on locally compact group $G$ were studied in \cite{CC} and \cite{ZZ}. It is easy to see that when $G$ is discrete, these weighted translations are special cases of pseudo-shifts. Theorefore by Theorem \ref{d-hyper1} and Corollary \ref{Pd-superbi}, we can also get the equivalent conditions for weighted translations on locally compact discrete groups to be disjoint hypercyclic and disjoint supercyclic.
\end{remark}
\section{Disjoint supercyclic operator weighted shifts on $\ell^2(\mathbb{Z,\mathcal{K}})$ }

  The bilateral operator weighted shifts on space $\ell^2(\mathbb{Z,\mathcal{K}})$ were studied by Hazarika and Arora in \cite{HA}. In \cite{WZ}, we proved that the bilateral operator weighted shifts are special cases of weighted pseudo-shifts. In this section, we will use Corollary \ref{Pd-superbi} to characterize the d-supercyclicity of bilateral operator weighted shifts on $\ell^2(\mathbb{Z,\mathcal{K}})$. First of all, let's recall some terminology.

 Let $\mathcal{K}$ be a separable complex Hilber space with an orthonormal basis $\{f_k\}_{k=0}^{\infty}.$ Define a separable Hilbert space
     $$\ell^2(\mathbb{Z,\mathcal{K}}):=\{x=(\ldots, x_{-1}, [x_0], x_1, \ldots): x_i\in \mathcal{K} \mbox{ and } \sum\limits_{i\in\mathbb{Z}}||x_i||^2<\infty\}$$
   under the inner product $\langle x, y\rangle=\sum\limits_{i\in\mathbb{Z}}\langle x_i, y_i\rangle_{\mathcal{K}}$.

   Let $\{A_{n}\}_{n=-\infty}^{\infty}$ be a uniformly bounded sequence of invertible positive diagonal operators on $\mathcal{K}$. The bilateral forward and backward  operator weighted shifts on $\ell^2(\mathbb{Z,\mathcal{K}})$ are defined as follows:

   $(i)$ The bilateral forward operator weighted shift $T$ on $\ell^2(\mathbb{Z,\mathcal{K}})$ is defined by
   $$ T(\ldots, x_{-1}, [x_0], x_1, \ldots)=(\ldots, A_{-2} x_{-2}, [A_{-1} x_{-1}], A_0 x_{0}, \ldots).$$
   Since $\{A_{n}\}_{n=-\infty}^{\infty}$  is uniformly bounded, $T$ is bounded and $||T||=\sup \limits_{i\in \mathbb{Z}}||A_i||<\infty.$ For $n > 0$,
  $$T^{n}(\ldots, x_{-1}, [x_0], x_1, \ldots)=(\ldots, y_{-1}, [y_0], y_1,\ldots),$$ where $y_j=\prod \limits_{s=0}^{n-1}A_{j+s-n}x_{j-n}.$

  $(ii)$ The  bilateral backward operator weighted shift $T$ on $\ell^2(\mathbb{Z,\mathcal{K}})$ is defined by
   $$ T(\ldots, x_{-1}, [x_0], x_1, \ldots)=(\ldots, A_{0} x_{0}, [A_{1} x_{1}], A_2 x_{2}, \ldots).$$

   Then
   $$T^{n}(\ldots, x_{-1}, [x_0], x_1, \ldots)=(\ldots, y_{-1}, [y_0], y_1, \ldots),$$
   where $y_j=\prod \limits_{s=1}^{n}A_{j+s}x_{j+n}.$

Since each $A_n$ is an invertible diagonal operator on $\mathcal{K},$ we conclude that
 \begin{eqnarray*}
||A_n||=\sup \limits_{k}||A_n f_k||\mbox{ and }  ||A_n^{-1}||=\sup \limits_{k}||A_n^{-1} f_k||.
\end{eqnarray*}

Now we are ready to state the main result.

 \begin{theorem}\label{forward}
Let $N \geq 2,$ for each $ 1 \leq l \leq N,$ let $T_l$ be a bilateral forward operator weighted shift on $\ell^2(\mathbb{Z,\mathcal{K}})$ with weight sequence $\{A_n^{(l)}\}_{n=-\infty}^{\infty},$  where $\{A_n^{(l)}\}_ {n=-\infty}^{\infty}$ is a uniformly bounded sequence of positive invertible diagonal operators on $\mathcal{K}.$
 Then for any integers $1\leq r_1 < r_2 < \cdots\cdots < r_{N},$ the following statements are equivalent:

\begin{enumerate}
\item $T_{1} ^{r_1}, T_{2} ^{r_2}, \ldots, T_{N} ^{r_N}$ are densely d-supercyclic.

\item There exists an increasing sequence $(n_k)_{k \geq 1}$ of positive integers such that:

      For every  $i_1, i_2\in \mathbb{N},$ $j_1, j_2\in \mathbb{Z}$ and
 $1 \leq l, s \leq N,$

   \begin{eqnarray*}
   \lim\limits_{k\rightarrow \infty}\left\|\left(\prod\limits_{v=j_1-r_ln_k}^{j_1-1}(A_v^{(l)})^{-1}\right)f_{i_1}\right\| \left\|\left(\prod\limits_{v=j_2}^{j_2+r_sn_k-1}A_v^{(s)}\right) f_{i_2}\right\|=0.
   \end{eqnarray*}

 For every  $i\in \mathbb{N},$ $j\in \mathbb{Z}$ and $1 \leq s < l \leq N,$
\begin{eqnarray*}
   \left\{\begin{array}{ll}
    \lim\limits_{k\rightarrow\infty}\left\|\left(\prod\limits_{v=j-r_sn_k}^{j-1}(A_v^{(s)})^{-1}\right)\left(\prod\limits_{v=j-r_sn_k}^{j+(r_l-r_s)n_k-1}A_v^{(l)}\right) f_{i}\right\|=0,\\
    \\
    \lim\limits_{k\rightarrow \infty}\left\|\left(\prod\limits_{v=j-r_ln_k}^{j-1}(A_v^{(l)})^{-1}\right)\left(\prod\limits_{v=j-r_ln_k}^{j-(r_l-r_s)n_k-1}A_v^{(s)} \right) f_{i}\right\|=0.
   \end{array}\right.
   \end{eqnarray*}
   \end{enumerate}
\end{theorem}

\begin{proof}
In \cite{WZ}, we proved that $\ell^2(\mathbb{Z,\mathcal{K}})$ is a Hilbert sequence space over $I:= \mathbb{N} \times \mathbb{Z}$ with an OP-basis $(e_{i,j})_{(i,j)\in I}.$ Where for $(i, j)\in I,$ $e_{i, j}:= (\ldots, z_{-1}, [z_0], z_1, \ldots) \in \ell^2(\mathbb{Z,\mathcal{K}})$ be defined by
\begin{eqnarray*}
z_{k} =\left\{\begin{array}{ll}
 f_{i} \;\;\;\; \mbox{if } k=j,\\
 \\
  0 \;\;\;\;\;\; \mbox{if } k \neq j.
  \end{array}\right.
\end{eqnarray*}

    By assumption, for each $1 \leq l \leq N$  $\{A_n^{(l)}\}_{n\in \mathbb{Z}}$ is a uniformly bounded sequence of positive invertible diagonal operators on $\mathcal{K},$ there exist a uniformly bounded sequence of positive sequences $\{(a_{i, n}^{(l)})_{i\in \mathbb{N}}\}_{n\in \mathbb{Z}}$
such that for each $n\in \mathbb{Z},$
$$A_{n}^{(l)}f_i = a_{i,n}^{(l)}f_i \;\mbox{ and }\; (A_{n}^{(l)})^{-1}f_i = (a_{i,n}^{(l)})^{-1}f_i  \;\;\;\mbox{  for  } i\in \mathbb{N}.$$
In this interpretation each $T_l$ is a weighted pseudo-shift $T_{b^{(l)}, \varphi}$ on $\ell^2(\mathbb{Z,\mathcal{K}})$ with
 $$b_{i,j}^{(l)} = a_{i, j-1}^{(l)}  \;\mbox{ and }\;  \varphi(i, j) = (i, j-1) \;\;\;\;\mbox{ for } (i, j) \in I.$$
 It follows from Corollary \ref{Pd-superbi} that $T_{1} ^{r_1}, T_{2} ^{r_2}, \ldots, T_{N} ^{r_N}$ are densely d-supercyclic if and only if there exists an increasing sequence $(n_k)$ of positive integers such that:

  For any $(i_1, j_1), (i_2, j_2) \in I$ and $1 \leq l, s\leq N,$
 \begin{eqnarray*}
&\;&\lim \limits_{k\rightarrow \infty} \left\|\left( \prod\limits_{v = 0}^{r_l n_k -1} b_{\varphi^v(i_1, j_1)} ^{(l)}\right)^{-1} e_{\varphi^{r_l n_k}(i_1, j_1)} \right\|
   \left\|\left( \prod\limits_{v = 1}^{r_s n_k} b_{\psi^v(i_2, j_2)} ^{(s)}\right) e_{\psi^{r_s n_k}(i_2, j_2)}\right\|\\
&\;&=\lim \limits_{k\rightarrow \infty}\left\| \left(\prod \limits_{v=0}^{r_ln_k-1}a_{(i_1, j_1-v-1)}^{(l)}\right)^{-1}e_{(i_1, j_1-r_ln_k)}\right\| \left\|\left(\prod \limits_{v=1}^{r_sn_k}a_{(i_2, j_2+v-1)}^{(s)}\right)e_ {(i_2, j_2+r_sn_k)}\right\| \\
&\;&=\lim\limits_{k\rightarrow \infty}\left\|\left(\prod\limits_{v=j_1-r_ln_k}^{j_1-1}(A_v^{(l)})^{-1}\right)f_{i_1}\right\| \left\|\left(\prod\limits_{v=j_2}^{j_2+r_sn_k-1}A_v^{(s)}\right) f_{i_2}\right\|=0.
\end{eqnarray*}

For every  $(i, j)\in I$ and $1 \leq s < l \leq N,$

\begin{eqnarray*}
&\;&\lim \limits_{k\rightarrow \infty} \left\|\left( \prod\limits_{v = 0}^{r_s n_k -1} b_{\varphi^v(i)} ^{(s)}\right)^{-1}
\left( \prod\limits_{v = 1}^{r_l n_k} b_{\psi^v (\varphi^{r_s n_k}(i))}^{(l)} \right) e_{\psi^{(r_l-r_s) n_k}(i)} \right\|\\
&\;&=\lim\limits_{k\rightarrow\infty}\left\|\left(\prod\limits_{v=j-r_sn_k}^{j-1}(A_v^{(s)})^{-1}\right)\left(\prod\limits_{v=j-r_sn_k}^{j+(r_l-r_s)n_k-1}A_v^{(l)}\right) f_{i}\right\|=0.
\end{eqnarray*}
and
\begin{eqnarray*}
&\;&\lim \limits_{k\rightarrow \infty} \left\|\left( \prod\limits_{v = 0}^{r_l n_k -1} b_{\varphi^v(i)} ^{(l)}\right)^{-1}
\left( \prod\limits_{v = 1}^{r_s n_k} b_{\psi^v (\varphi^{r_l n_k}(i))} ^{(s)}\right) e_{\varphi^{(r_l-r_s) n_k}(i)}\right\|\\
&\;&= \lim\limits_{k\rightarrow \infty}\left\|\left(\prod\limits_{v=j-r_ln_k}^{j-1}(A_v^{(l)})^{-1}\right)\left(\prod\limits_{v=j-r_ln_k}^{j-(r_l-r_s)n_k-1}A_v^{(s)} \right) f_{i}\right\|=0.
\end{eqnarray*}
 Which completes the proof.
\end{proof}

 In \cite{FS}, Feldman considered the hypercyclicity of bilateral weighted shifts on $\ell^2(\mathbb{Z})$ that are invertible. Motivated by Feldman's work, in \cite{WZ} we showed that if the weight sequence $\{A_n\}_{n=-\infty}^{\infty}$ are assumed to satisfy that there exists some $m > 0$ such that $||A_n^{-1}||\leq m$ for all $n < 0$ (or for all $n > 0$), then the characterizing conditions for d-hypercyclicity simplify. Now, we establish the d-supercyclic conditions for this special case. The proof is similar to that in \cite{WZ}, here we omit it.

\begin{corollary}\label{same}
Let $T$ be a bilateral forward operator weighted shift on $\ell^2(\mathbb{Z,\mathcal{K}})$ with weight sequence $\{ A_n\}_{n=-\infty}^{\infty},$ where $\{A_n\}_{n=-\infty}^{\infty}$ is a uniformly bounded sequence of positive invertible diagonal operators on $\mathcal{K},$ and there exists $m > 0$ such that $||A_n^{-1}||\leq m$ for all $n < 0$ (or for all $n > 0$).
Then for any integer $N\geq 2,$ the following are equivalent:

\begin{enumerate}
\item $T, T^{2}, \ldots, T^N$ are densely d-supercyclic.

\item There exists an increasing sequence $(n_k)_{k\geq 1}$ of positive integers such that:

For every $i_1, i_2 \in \mathbb{N}$,
\begin{eqnarray*}
 \lim\limits_{k\rightarrow \infty}\left\|\left(\prod\limits_{v=1}^{ln_k}(A_{-v})^{-1}\right) f_{i_1}\right\| \left\|\left(\prod\limits_{v=1}^{N n_k}A_{v}\right) f_{i_2}\right\| = 0 \;\;\;(1 \leq l \leq N).
 \end{eqnarray*}
 and
\begin{eqnarray*}
 \lim\limits_{k\rightarrow \infty}\left\|\left(\prod\limits_{v=1}^{N n_k}(A_{-v})^{-1}\right) f_{i_1}\right\| \left\|\left(\prod\limits_{v=1}^{l n_k}A_{v}\right) f_{i_2}\right\| = 0 \;\;\;(1 \leq l \leq N).
 \end{eqnarray*}
 For every $i\in \mathbb{N}$,
\begin{eqnarray*}
 \lim\limits_{k\rightarrow \infty}\left\|\left(\prod\limits_{v=1}^{ln_k}(A_{-v})^{-1}\right) f_{i}\right\| =  \lim\limits_{k\rightarrow \infty}\left\|\left(\prod\limits_{v=1}^{ln_k}A_{v}\right) f_{i}\right\| = 0 \;\;\;(1 \leq l \leq N-1).
 \end{eqnarray*}
\end{enumerate}
\end{corollary}

\begin{example}
For each $s\in \mathbb{N},$ let
 $$ \mathcal{C}_s = \{ 2^{2s+1}-(2s+1),\ldots, 2^{2s+1}-1\}, $$
 $$\mathcal{D}_s =\{2^{2s+1} ,\ldots, 2^{2s+1}+(2s+1)-1\}$$
 and
 $$ \mathcal{C} = \bigcup\limits_{s=0}^\infty \mathcal{C}_s, \; \mathcal{D} = \bigcup\limits_{s=0}^\infty \mathcal{D}_s, \; \mathcal{E} =\bigcup\limits_{s=0}^\infty \{-2^{2s+1}\}.$$

 Let $\{ A_n\}_{n=-\infty}^{\infty}$ be a uniformly bounded sequence of positive invertible diagonal operators on $\mathcal{K},$ defined as follows:
 \begin{eqnarray*}
&\mbox{If }& n\in \mathcal{C} : A_n(f_k)=\left\{\begin{array}{ll} \frac{1}{2}f_k,\;\;\;\;0\leq k \leq n,\\
\\
  f_k,\;\;\;\;\;k > n.
 \end{array}\right.\\
 \\
 &\mbox{If }& n\in \mathcal{D} : A_n(f_k)=\left\{\begin{array}{ll} 2f_k,\;\;\;\;0\leq k \leq n,\\
\\
  f_k,\;\;\;\;\;k > n.
 \end{array}\right.\\
 \\
&\mbox{If }& n\in \mathcal{E}  : A_n(f_k)=\left\{\begin{array}{ll} 2f_k,\;\;\;\;0\leq k \leq -n,\\
\\
  f_k,\;\;\;\;\;k > -n.
 \end{array}\right.\\
\\
&\mbox{If }& n\in \mathbb{Z}\backslash (\mathcal{C}\cup \mathcal{D}\cup\mathcal{E}) :
A_n(f_k) = f_k\;\;\;\;\mbox{for all } k\geq 0.
\end{eqnarray*}
Let $T$ be the bilateral forward operator weighted shift on $\ell^2(\mathbb{Z,\mathcal{K}})$ with weight sequence $\{ A_n\}_{n=-\infty}^{\infty}.$  Then $T, T^{2}$ are d-supercyclic.
\end{example}

\begin{proof}
Notice that for any $n\in \mathbb{Z},$ $\frac{1}{2} \leq \|A_n\| \leq 2,$ we use Corollary \ref{same} to give the proof. Let $(n_k)_{k\geq 1} = (2^{2k+1})_{k\geq 1}.$ Then for each $i\in \mathbb{N},$
\begin{eqnarray*}
 \left\|\left(\prod\limits_{v=1}^{n_k}(A_{-v})^{-1}\right) f_{i}\right\|\leq (\frac{1}{2})^{k-i+1}\rightarrow 0 \mbox{ as } k\rightarrow \infty
  \end{eqnarray*}
and
  \begin{eqnarray*}
  \left\|\left(\prod\limits_{v=1}^{n_k}A_{v}\right) f_{i}\right\|\leq (\frac{1}{2})^{2k-i}\rightarrow 0 \mbox{ as } k\rightarrow \infty.
 \end{eqnarray*}
Since $2^{2k+1}+(2k+1)-1< 2\cdot2^{2k+1}<2^{2k+3}-(2k+3),$
\begin{eqnarray}\label{4.1}
 (\frac{1}{2})^i\leq \left\|\left(\prod\limits_{v=1}^{2n_k}A_{v}\right) f_{i}\right\|\leq (\frac{1}{2})^{2k-i}\leq 2^i,
 \end{eqnarray}
 hence for any $i_1, i_2 \in I,$
 \begin{eqnarray*}
 \lim\limits_{k\rightarrow \infty}\left\|\left(\prod\limits_{v=1}^{n_k}(A_{-v})^{-1}\right) f_{i_1}\right\| \left\|\left(\prod\limits_{v=1}^{2 n_k}A_{v}\right) f_{i_2}\right\|=0,
 \end{eqnarray*}
 also by the fact $2^{2k+1}< 2\cdot2^{2k+1}<2^{2k+3},$ it is easy to see that
 \begin{eqnarray*}
 \lim\limits_{k\rightarrow \infty}\left\|\left(\prod\limits_{v=1}^{2n_k}(A_{-v})^{-1}\right) f_{i_1}\right\| \left\|\left(\prod\limits_{v=1}^{n_k}A_{v}\right) f_{i_2}\right\| = 0
 \end{eqnarray*}
 and
 \begin{eqnarray*}
 \lim\limits_{k\rightarrow \infty}\left\|\left(\prod\limits_{v=1}^{2n_k}(A_{-v})^{-1}\right) f_{i_1}\right\| \left\|\left(\prod\limits_{v=1}^{2n_k}A_{v}\right) f_{i_2}\right\| = 0.
 \end{eqnarray*}
 It follows from Corollary \ref{same}, $T, T^2$ are d-supercyclic.

But it follows from \eqref{4.1} that $\left\|\left(\prod\limits_{v=1}^{2n_k}A_{v}\right) f_{i}\right\| \nrightarrow 0 \mbox{ as } k\rightarrow \infty.$ By Theorem  \ref{d-hyper1} $T, T^{2}$ are not d-hypercyclic.
\end{proof}

\bibliographystyle{amsplain}

\begin{thebibliography}{99}

\bibitem{B1} L. Bernal-Gonz\'{a}lez, \textit{Disjoint hypercyclic operators}, Studia Math. \textbf{182} (2007), no. 2, 113-131.


\bibitem{BM} F. Bayart and \'{E}. Matheron, \textit{Dynamics of Linear Operators}, Cambridge University Press, 2009.

\bibitem{BMP} J. B\`{e}s, \"{O}. Martin and A. Peris, \textit{Disjoint  hypercyclic linear fractional composition operators}, J. Math. Anal. Appl. \textbf{381} (2011), 843-856.

\bibitem{BMS} J. B\`{e}s, \"{O}. Martin and R. Sanders, \textit{Weighted shifts and disjoint hypercyclicity}, J. Operator Theory, \textbf{72} (2014), no. 2, 15-40.

\bibitem{BP} J. B\`{e}s and A. Peris, \textit{Disjointness in hypercyclicity}, J. Math. Anal. Appl. \textbf{336} (2007), 297-315.


\bibitem{BMPS} J. B\`{e}s, \"{O}. Martin, A. Peris and S. Shkarin, \textit{Disjoint mixing operators}, J. Funct. Anal. \textbf{263} (2012), 1283-1322.

 \bibitem{CC} C. C. Chen, \textit{Disjoint hypercyclic weighted translations on groups}, Banach J. Math. Anal. \textbf{11} (2017), no. 3, 459-476.

 \bibitem{FS} N. S. Feldman, \textit{Hypercyclicity and supercyclicity for invertible bilateral weighted shifts}, Proc. Amer. Math. Soc. \textbf{131} (2003), no. 2, 479-485.

\bibitem{GM} K.-G. Grosse-Erdmann and A. Peris Manguillot, \textit{Linear Chaos}, Springer, New York, 2011.

\bibitem{GE} K.-G. Grosse-Erdmann, \textit{Hypercyclic and chaotic weighted shifts}, Studia Math. \textbf{139} (2000), 47--68.

\bibitem{HA} M. Hazarika and S. C. Arora, \textit{Hypercyclic operator weighted shifts}, Bull. Korean. Math. Soc. \textbf{41} (2004), 589-598.




\bibitem{LZ} Y. X. Liang and Z. H. Zhou, \textit{Disjoint supercyclic powers of weighted shifts on weighted sequence spaces}, Turk. J. Math. \textbf{38} (2014), 1007-1022.

\bibitem{LZ2} Y. X. Liang and Z. H. Zhou, \textit{Hereditarily hypercyclicity and supercyclicity of weighted shifts}, J. Korean Math. Soc. \textbf{51} (2014), no. 2, 363-382.

\bibitem{MO} \"{O}. Martin, \textit{Disjoint hypercyclic and supercyclic composition operators}, PhD, Bowling Green State University, Bowling Green, 2011.

\bibitem{MRA}  R. A. Mart\'{\i}nez-Avenda\~{n}o, \textit{Hypercyclicity of shifts on weighted $L^p$ spaces of directed trees}, J. Math. Anal. Appl. \textbf{446} (2017), no. 1, 823-842.

\bibitem{Sh} S. Shkarin, \textit{A short proof of existence of disjoint hypercyclic operators}, J. Math. Anal. Appl. \textbf{367} (2010), 713-715.

\bibitem{SA} H. N. Salas, \textit{Dual disjoint hypercyclic operators}, J. Math. Anal. Appl. \textbf{374} (2011), no. 1, 106-117.

\bibitem{S1} H. N. Salas, \textit{Hypercyclic weighted shifts}, Trans. Amer. Math. Soc. \textbf{347} (1995), 993-1004.

\bibitem{SH}  H. N. Salas, \textit{Supercyclicity and weighted shifts}, Studia Math. \textbf{135} (1999), 55-74.

 \bibitem{WZ} Y. Wang and Z. H. Zhou, \textit{Disjoint hypercyclic powers of weighted pseudo-shifts}, Bull. Malays. Math. Sci. Soc. (2017), 1-20. https://doi.org/10.1007/s40840-017-0584-7

\bibitem{ZZ} L. Zhang, H. Q. Lu, X. M. Fu and Z. H. Zhou, \textit{Disjoint hypercyclic powers of weighted translations on groups}, J. Math. Czechoslovak. \textbf{67} (2017), no. 3, 839-853.




\end{thebibliography}

\end{document}